\documentclass[12pt,english,a4paper]{amsart}
\usepackage[T1]{fontenc}
\usepackage[latin1]{inputenc}
\usepackage[a4paper]{geometry}
\usepackage{amstext}
\usepackage{amsthm}
\usepackage{amssymb}

\makeatletter

\let\SF@@footnote\footnote
\def\footnote{\ifx\protect\@typeset@protect
    \expandafter\SF@@footnote
  \else
    \expandafter\SF@gobble@opt
  \fi
}
\expandafter\def\csname SF@gobble@opt \endcsname{\@ifnextchar[
  \SF@gobble@twobracket
  \@gobble
}
\edef\SF@gobble@opt{\noexpand\protect
  \expandafter\noexpand\csname SF@gobble@opt \endcsname}
\def\SF@gobble@twobracket[#1]#2{}

\numberwithin{equation}{section}
\numberwithin{figure}{section}
\theoremstyle{plain}
\newtheorem{thm}{\protect\theoremname}[section]
  \theoremstyle{plain}
  \newtheorem{cor}[thm]{\protect\corollaryname}
  \theoremstyle{remark}
  \newtheorem{rem}[thm]{\protect\remarkname}
  \theoremstyle{definition}
  \newtheorem{defn}[thm]{\protect\definitionname}
  \theoremstyle{plain}
  \newtheorem{prop}[thm]{\protect\propositionname}
  \theoremstyle{plain}
  \newtheorem{lem}[thm]{\protect\lemmaname}
  \theoremstyle{plain}
  \newtheorem{conjecture}[thm]{\protect\conjecturename}

\def\makebbb#1{
    \expandafter\gdef\csname#1\endcsname{
        \ensuremath{\Bbb{#1}}}
}
\makebbb{R}
\makebbb{N}
\makebbb{Z}
\makebbb{C}
\makebbb{G}
\makebbb{E}
\makebbb{H}
\makebbb{P}
\makebbb{B}
\makebbb{K}
\makebbb{E}

\makeatother

\usepackage{babel}
  \providecommand{\conjecturename}{Conjecture}
  \providecommand{\corollaryname}{Corollary}
  \providecommand{\definitionname}{Definition}
  \providecommand{\lemmaname}{Lemma}
  \providecommand{\propositionname}{Proposition}
  \providecommand{\remarkname}{Remark}
\providecommand{\theoremname}{Theorem}

\begin{document}

\title{determinantal point processes and fermions on polarized complex manifolds:
bulk universality }

\author{Robert J. Berman}

\email{robertb@chalmers.se}
\begin{abstract}
We consider determinantal point processes on a compact complex manifold
$X$ in the limit of many particles. The correlation kernels of the
processes are the Bergman kernels associated to a a high power of
a given Hermitian holomorphic line bundle $L$ over $X.$ The empirical
measure on $X$ of the process, describing the particle locations,
converges in probability towards the pluripotential equilibrium measure,
expressed in term of the Monge-Ampère operator. The asymptotics of
the corresponding fluctuations in the bulk are shown to be asymptotically
normal and described by a Gaussian free field and applies to test
functions (linear statistics) which are merely Lipschitz continuous.
Moreover, a scaling limit of the correlation functions in the bulk
is shown to be universal and expressed in terms of (the higher dimensional
analog of) the Ginibre ensemble. This geometric setting applies in
particular to normal random matrix ensembles, the two dimensional
Coulomb gas, free fermions in a strong magnetic field and multivariate
orthogonal polynomials.
\end{abstract}

\maketitle
\tableofcontents{}

\section{Introduction}

The systematic study of \emph{determinantal point processes} was initiated
by Macchi \cite{m} in the seventies who called them \emph{fermionic}
point processes, inspired by the properties of fermion gases in statistical
(quantum) mechanics. For general reviews see \cite{s1,h-k-p,j2}.
The theory concerns ensembles of ``particle configurations'' on
a given space $X$ which exhibit repulsion. An important class of
such processes are the determinantal projection processes, which may
be defined by a probability measure on the $N-$fold product $X^{N},$
the ''configuration space of $N$ particles on $X",$ with the property
that its density may be written as 
\begin{equation}
\mathcal{\rho}^{(N)}(x_{1},...,x_{N})=\frac{1}{N!}\det(\mathcal{K}(x_{i},x_{j})),\label{eq:intro general prob density}
\end{equation}
 where the kernel $\mathcal{K}$ is the integral kernel of an orthogonal
projection operator onto a vector space of dimension $N.$ As a consequence
the probability distributions vanish for a configuration $(x_{1},...,x_{N})$
of points $x_{i}$ as soon as two points coincide, explaining the
repulsive behavior of the ensemble. As it turns out, in many situations
such ensembles are \emph{critical} in the sense that they naturally
appear in sequences with $N,$ the number of particles, tending to
infinity in such a way that a well-defined limiting ensemble may be
extracted. Moreover, large classes of such sequences of ensembles
often give rise to one and the same limit. This is the phenomenon
of \emph{universality} (see \cite{dei1} for a nice survey). Perhaps
its most famous illustration is given by ensembles of $N\times N$
\emph{Hermitian random matrices} whose eigenvalues, in the large $N$
limit, determine a unique determinantal point process on the real
line. This latter process has also been conjectured to describe the
statistics of the zeroes of the Riemann zeta function, as well as
statistics of quantum systems whose classical dynamics is chaotic(
references and more recent relations to random growth and tiling problems
may be found in \cite{j2}). 

The present paper concerns a general class of such critical ensembles,
where the space $X$ is a compact complex manifold equipped with an
holomorphic line bundle $L$ with a given Hermitian metric locally
represented as $e^{-\phi},$ where $\phi$ is called a ``weight''
on $L.$ The kernel $\mathcal{K}$ defining the ensemble may then
be identified with the orthogonal projection onto the space of global
holomorphic sections $H^{0}(X,L)$ of $L$ (with respect to a local
unitary frame of $(L,e^{-\phi}))$ and the corresponding determinantal
probability density on $X^{N}$ may  be written as the squared point-wise
norm of the normalized Vandermonde type determinant $(\det S)(x_{1},...,x_{N})$
associated to any given base $S=(s_{1},...,s_{N})$ of sections in
$H^{0}(X,L):$ 
\begin{equation}
\mathcal{\rho}^{(N)}(x_{1},...,x_{N})=\frac{1}{\mathcal{Z}_{N}}\left|\det(S)(x_{1},...,x_{N})\right|_{\phi}^{2}\label{eq:slater det intro}
\end{equation}
In this setting the limit of a large number $N$ of particles corresponds
to the limit when the line bundle $L$ is replaced by a large tensor
power, written as $kL$ in additive notation. When $X$ is the complex
projective space this setting is just a geometric formulation of the
theory of (weighted) multivariate orthogonal polynomials, with the
tensor power $k$ corresponding to the degree of the polynomials (see
section \ref{sec:Examples}). In mathematical physics terminology
$H^{0}(X,L)$ may be identified with the quantum ground state space
of a single fermion (complex spinor) on $X$ subject to an exterior
magnetic field and the density in formula \ref{eq:slater det intro}
is the squared probability amplitude for the corresponding maximally
filled many particle state, i.e. $(\det S)$ is the corresponding
Slater determinant.

Already in the simplest case when $X$ is the complex projective line,
i.e. the Riemann sphere (viewed as the one-point compactification
of $\C)$ the corresponding ensemble is remarkably rich and admits
at least three different well-known descriptions in terms of \emph{(1)
normal random matrices, (2) a free fermion gas,} \emph{(3) a Coulomb
gas} of repelling electric charges \cite{za}. Compare the discussion
in Section \ref{sec:Examples}.

While there are quite recent result concerning this special case,
both in mathematics and physics, there seems to be almost no previous
general results in the higher dimensional situation studied in the
present paper. For one reference see the recent paper \cite{t-s-z}.
As it turns out, the main new feature that appears in higher dimensions
is that the role of the Laplace operator in one complex dimension
(which expresses the limiting expected density of particles) is played
by the \emph{fully non-linear} \emph{Monge-Ampère operator,} which
is the subject of (complex) \emph{pluripotential} theory \cite{kl,g-z}.
In fact, one of the motivations for the present paper and the companion
paper \cite{berman ldp} is to develop a Coulomb gas type descriptions
of a gas of free fermions on complex manifolds and conversely to provide
a statistical mechanical interpretation of complex pluripotential
theory. An important feature of our approach is that it does not require
that $\phi$ be positively curved, i.e. that the corresponding magnetic
two-form has any definite sign properties. As will be explained below
this means that the support of the limiting one-point correlation
functions will only cover a proper subset $D$ of $X,$ which corresponds
to the droplet appearing in the physical description of the Quantum
Hall Effect (QHE) describing fermions in large magnetic fields \cite{lau}.
We will here focus on the universality properties in the ``bulk''
of the droplet $D$ leaving the case of the boundary (edge) properties
as challenging open problem for the future (which from a physical
point of view can be expected to be related to the properties of the
edge states playing a central role in the QHE). 

Yet another motivation comes from approximation theory where configurations
$(x_{1},...x_{N})$ appear as \emph{interpolation nodes} on $X$ and
a configuration maximizing a functional of the form \ref{eq:intro general prob density}
is known to have optimal interpolation properties in a certain sense
\cite{g-m-s,s-w}. Sequences of such configurations, with $N$ tending
to infinity, then appear naturally in discretization schemes. Moreover,
as shown very recently in \cite{b-b-w} any such optimal sequence
equidistributes asymptotically on the corresponding equilibrium measure.
This fact should be compared with Theorem \ref{thm:conv in prob}
in the present paper which shows that, with high probability, the
same equidistribution property holds for \emph{random} configurations
of the corresponding ensemble.

One final motivation comes from the study by Shiffman, Zelditch and
coworkers of random zeroes of holomorphic sections of positive line
bundles, where many statistical results have been obtained and where
a key role is played by Bergman kernels (cf. \cite{b-s-z,sz2,sz4}).

\subsection{Statement of the main results}

Let $L$ be a holomorphic line bundle over a compact complex manifold
$X.$ Denote by $H^{0}(X,L)$ the vector space of all global holomorphic
sections on $X$ with values in $L$ and write $N:=\dim H^{0}(X,L).$
Fixing an Hermitian metric on $L$ (locally represented by $e^{-\phi}$(
where the additive object $\phi$ is called a \emph{weight} $\phi)$
and a suitable measure $\mu$ on $X$ induces an inner product on
$H^{0}(X,L)$ defined by 
\[
\left\Vert s\right\Vert _{\phi}^{2}:=\int_{X}\left|s\right|^{2}e^{-\phi}\mu
\]
(abusing notation slightly; see section \ref{sub:Notation-and-general}).
We will denote the corresponding Hilbert space by $\mathcal{H}(X,L)$
and its \emph{Bergman kernel} by $K,$ which is the integral kernel
of the orthogonal projection $\mathcal{C^{\infty}}(X,L)\rightarrow H^{0}(X,L):$
\begin{equation}
K(x,y)=\sum_{i=1}^{N}s_{i}(x)\otimes\overline{s_{i}(y)},\label{def: K}
\end{equation}
 where $(s_{i})$ is an orthonormal bases in $\mathcal{H}(X,L).$

As is essentially well-known this setup induces a probability measure
$\gamma_{P}$ on the $N-$fold product $X^{N}$ whose density (w.r.t.
$\mu^{\otimes N})$ is defined as the determinant of an $N\times N$
matrix:
\begin{equation}
\rho^{(N)}(x_{1},...,x_{N}):=\frac{1}{N!}\det(K(x_{i},x_{j})e^{-\frac{1}{2}(\phi(x_{i})+\phi(x_{j}))}),\label{eq:intro prob measure}
\end{equation}
The main object of study in the present paper is the large $k$ asymptotics
of the probability space $(X^{N},\gamma_{P}),$ when $L$ is replaced
by its $k$th tensor power (written as $kL$ in our additive notation)
equipped with the induced weight $k\phi.$ In the following a subindex
$k$ will be used to indicate the the dependence on the parameter
$k.$ We will always assume that $L$ is \emph{big,} i.e that 
\[
N_{k}:=\dim H^{0}(X,kL)=Vk^{n}+o(k^{n-1}),\,\,\,V>0
\]
 (where the constant $V$ is usually called the \emph{volume} of $L).$
The main case of interest appears when $L$ is (very) ample, so that
$X$ may be embedded as algebraic manifold in complex projective space
and $L$ is the restriction of the hyperplane line bundle. Then $(X,L)$
is called a polarized manifold and $H^{0}(X,kL)$ gets identified
with the restriction to $X$ of the space of all homogeneous polynomials
of degree $k.$ Moreover, the main results in the present paper concern
weighted measured $(\phi,\mu)$ which for which we introduce the (non-standard)
terminology\emph{ strongly regular}. This will mean that the weight
$\phi$ is locally $\mathcal{C}^{1,1}$-smooth, i.e. it is differentiable
and all of its first partial derivatives are locally Lipschitz continuous,
and the measure $\mu=\omega_{n}$ is the volume form of a continuous
metric $\omega$ on $X.$ The reason that we assume that $\phi$ is
merely $\mathcal{C}^{1,1}$-smooth, rather than $\mathcal{C}^{2}-$smooth
(or even $\mathcal{C^{\infty}}-$smooth) is that this appears to be
the essentially optimal regularity class where the results below concerning
universality of the scaled correlation functions can be expected to
hold. Moreover, since $\phi$ is not assumed to be positively curved
we will anyway have to work with the corresponding equilibrium weight
$\phi_{e}$ in the proofs which is almost never $\mathcal{C}^{2}-$smooth,
even if $\phi$ is smooth (unless $\phi$ is positively curved; compare
\cite{berm45}). When $X$ is the complex projective space $X:=\P^{n}$
and $L$ the hyperplane line bundle $\mathcal{O}(1)$ (so that $H^{0}(X,kL)$
may be identified with the space of all polynomials of total degree
at most $k$ in $\C^{n})$ we also allow $\omega_{n}$ to be the Lebesgue
measure on the affine piece $\C^{n}$ as long as $\phi$ has super
logarithmic growth (formula \ref{eq:ass on growth of phi}). 

The notion of strongly regular weighted measures $(\phi,\mu)$ on
$X$ that we shall focus on in the present paper should be contrasted
with the considerably more general notion of weighted measures $(\phi,\mu)$
satisfying the \emph{Bernstein-Markov property} in the sense of \cite{b-b-w}.
From the probabilistic point of view the latter property simply means
that the one-point correlation function $\rho_{k}^{(1)}$ of the corresponding
determinantal point process has sub-exponential growth in $k.$ For
example, the Bernstein-Markov property is satisfied if $\phi$ is
continuous and $\mu$ is a continuous volume form on a complex or
real algebraic variety. In particular, the latter property applies
when $\mu$ is Lebesgue measure on $\R^{n},$ as in the setting of
Hermitian random matrices \cite{dei2} (where $n=1)$. 

As a guide line, the  Bernstein-Markov property of $(\phi,\mu)$ is
enough to establish asymptotics in the ``macroscopic regime'', such
as convergence in probability towards the corresponding equilibrium
measure. In contrast, the results in the ``microscopic regime'',
concerning length scales of the order $k^{-1/2}$ on $X,$ only hold
in the strongly regular case.

\subsubsection{Correlation functions and the equilibrium measure}

As is well known all the \emph{$m-$point correlation functions} $\rho_{k}^{(m)}$,
where $1\leq m\leq N_{k},$ of the ensemble above may be expressed
as (weighted) determinants of $K_{k}(x_{i},x_{j}).$ In particular,
\[
\rho_{k}^{(1)}(x)=K_{k}(x,x)e^{-k\phi(x)},\,\,\,\,\,\rho^{(2).c}(x,y)=-\left|K_{k}(x,y)\right|^{2}e^{-k\phi(x)}e^{-k\phi(y)},
\]
 where $\rho^{(2).c}$ is the \emph{connected} 2-point correlation
function (see section \ref{sub:Correlation-functions}). As shown
in \cite{berm45}, in the strongly regular case,
\begin{equation}
\frac{1}{N_{k}}\rho_{k}^{(1)}\omega_{n}\rightarrow\mu_{\phi_{e}},\label{eq:intro conv towards equ}
\end{equation}
weakly, when $k\rightarrow\infty,$ where $\mu_{\phi_{e}}$ is the
pluripotential \emph{equilibrium measure} (of $(X,\phi)),$ which
may be written as the Monge-Ampère measure $\frac{1}{Vn!}(dd^{c}\phi_{e})^{n}$
of the equilibrium weight $\phi_{e}$ and represented as 
\[
\frac{1}{Vn!}(dd^{c}\phi_{e})^{n}=1_{S}\det_{\omega}(dd^{c}\phi)(x)\frac{\omega^{n}}{Vn!},
\]
 where $S\subset X$ denotes the support of the equilibrium measure
(see Section \ref{sec:The-pluripotential-equilibrium}). We recall
that in the case of one complex dimension (i.e. $n=1)$ the support
$S$ is referred to as the \emph{droplet }in the physics literature
on the Quantum Hall Effect (see \cite{lau,za} and Section \ref{sec:Examples}
below). 

As later shown in \cite{b-b-w} the convergence \ref{eq:intro conv towards equ}
holds, in the weak topology, for weighted measures $(\phi,\mu)$ satisfying
the Bernstein-Markov property. However, in the strongly regular setting
that we will concentrate on here \emph{point-wise} convergence actually
holds in the sense that there is a subset of $X$ that will be called
the \emph{weak bulk} (of $(X,\phi)$) such that
\[
\frac{1}{N_{k}}\rho_{k}^{(1)}(x)\rightarrow\frac{1}{V}\det_{\omega}(dd^{c}\phi)(x),\,\,\,\,x\,\textrm{in\,\ the\,\ weak bulk}
\]
and converges to zero almost everywhere in the complement of the weak
bulk. We recall that in the random matrix and Coulomb gas literature
the bulk of the equilibrium measure is usually defined as the interior
of the support $S$ of the equilibrium measure. But the problem is
that, for a general smooth weight $\phi,$ the set $S$ may be extremely
irregular and, a priori, its interior could be empty. In contrast,
the weak bulk always has positive Lebesgue measure. The precise definition
of the weak bulk is given in section \ref{sec:The-pluripotential-equilibrium}
and uses that, by the results in \cite{berm45}, the equilibrium weight
$\phi_{e}$ is $C^{1,1}-$smooth and hence the second derivatives
exist almost everywhere.

The following theorem gives the scaling asymptotics of the Bergman
kernel, around a fixed point $x$ in the weak bulk. It is expressed
in terms of ``normal'' local coordinates $z$ centered at $x$ and
a ``normal'' trivialization of $L,$ i.e such that 
\begin{equation}
\,\,\,\omega(z)=\frac{i}{2}\sum_{i=1}^{n}dz_{i}\wedge\overline{dz_{i}}+....,\,\,\,\phi(z)=\sum_{i=1}^{n}\lambda_{i}\left|z_{i}\right|^{2}+...\label{eq: Intro metrics}
\end{equation}
where the dots indicate ``higher order terms''. Hence, $\lambda_{i}$
are the eigenvalues of the curvature form $dd^{c}\phi$ w.r.t the
metric $\omega$ and we denote the corresponding diagonal matrix by
$\lambda.$
\begin{thm}
\label{thm:intro bulk univ}Assume that the weight $\phi$ is in $C_{loc}^{1,1}$
and that the volume form $\omega_{n}$ is continuous. Let $x$ be
a fixed point in the weak bulk and take ``normal'' local coordinates
$z$ centered at $x$ and a ``normal'' trivialization of $L$ as
above. Then 
\begin{equation}
k^{-n}K_{k}(k^{-1/2}z,k^{-1/2}w)\rightarrow\frac{\det\lambda}{\pi^{n}}e^{\left\langle \lambda z,w\right\rangle }\label{eq:Bergman kernel asympt in theorem}
\end{equation}
in the $\mathcal{C}^{\infty}-$topology on compact subsets of $\C_{z}^{n}\times\C_{w}^{n}.$
In particular, the connected 2-point function has the following scaling
asymptotics

\[
-k^{-2n}\rho_{k}^{(2).c}(k^{-1/2}z,k^{-1/2}w)\rightarrow(\frac{\det\lambda}{\pi^{n}})^{2}e^{-\sum_{i=1}^{n}\lambda_{i}\left|z_{i}-w_{i}\right|^{2}}
\]
uniformly on compacts of $\C_{z}^{n}\times\C_{w}^{n}.$
\end{thm}
In the case when $\phi$ is $C^{\infty}-$smooth and strictly positively
curved (and in particular the weak bulk coincides with all of $X)$
the convergence \ref{eq:Bergman kernel asympt in theorem} was shown
in \cite{b-s-z}, where it was deduced from the microlocal analysis
of the Boutet de Monvel-Sjöstrand parametrix for the corresponding
Szegö kernel \cite{b-s} following \cite{ze} (which also yields an
explicit control on the remainder terms). As emphasized in \cite{b-s-z}
the previous theorem may on one hand be interpreted as a ``localization''
result, in the sense that the limit is expressed in terms of local
data (the curvature of $dd^{c}\phi$ at the fixed point). On the other
hand, it can be seen as a ``universality'' result (see \cite{dei1}
for a general discussion of universality in mathematics and physics).
Indeed, scaling the coordinates further in order to make the Kähler
metric $dd^{c}\phi$ at the fixed point the ``yard stick'' the limiting
kernel becomes independent of the ensemble (and coincides with the
Bergman kernel of Fock space). When $n=1$ the corresponding limiting
one-dimensional determinantal point process was studied by Ginibre,
who showed that it appears from a scaling limit of random complex
matrices with independent complex Gaussian entries. 

As a corollary the following analog of a well-known universality result
for the Hermitian random matrix model (where the limiting kernel is
the sine kernel) is obtained:
\begin{cor}
\label{cor:normal}Let $\phi$ be a function in $C_{loc}^{1,1}(\C)$
with super logarithmic growth and denote by $\rho_{k}^{(\cdot).\cdot}$
the eigenvalue correlation functions of the associated normal random
matrix model (see section \ref{sub:Normal-matrices}). Then the following
convergence holds when the rank $N=k+1$ of the matrices tends to
infinity:
\[
-\frac{\rho_{k}^{(2).c}(z_{0}+\frac{z}{\sqrt{\rho_{k}^{(1)}(z_{0})}},z_{0}+\frac{w}{\sqrt{\rho_{k}^{(1)}(z_{0})}},)}{\left(\rho_{k}^{(1)}(z_{0})\right)^{2}}\rightarrow e^{-\left|z-w\right|^{2}}
\]
 uniformly on compacts of $\C\times\C,$ when $z_{0}$ is a fixed
point in the weak bulk (in the eigenvalue plane $\C).$
\end{cor}
The remaining main results concern properties inside the\emph{ bulk}
of $(X,\phi)$ which, when the weight $\phi$ is $C^{2}-$smooth,
is defined as the interior of the support $S$ of the equilibrium
measure. In general, the bulk (which always contains the weak bulk
appearing above) is defined as the largest open subset of $S$ where
\begin{equation}
\omega_{\phi}:=dd^{c}\phi\label{eq:K=0000E4hler metric in bulk intro}
\end{equation}
 defines a continuous Kähler metric (i.e. a continuous strictly positive
form). The next theorem implies that the correlations are short range
on macroscopic length scales in the bulk:
\begin{thm}
\label{thm:intro decay}Assume that the weight $\phi$ is in $C_{loc}^{1,1}$
and that the volume form $\omega_{n}$ is continuous. Let $E$ be
a compact subset of of the bulk. Then there is a constant $C$ (depending
on $E)$ such that the following estimate holds for all pairs $(x,y)$
such that either $x$ or $y$ is in $E:$
\[
-k^{-2n}\rho_{k}^{(2).c}(x,y)\leq Ce^{-\sqrt{k}d(x,y)/C}
\]
for all $k,$ where $d(x,y)$ is the distance function with respect
to a fixed smooth metric on $X.$ 
\end{thm}

\subsubsection{Fluctuations of linear continuous statistics}

Consider the random measure (i.e. a measure valued random variable)
defined by 
\begin{equation}
(x_{1},...,x_{N})\mapsto\sum_{i=1}^{N}\delta_{x_{i}},\label{eq:intro random measure}
\end{equation}
Its expected value is the one point correlation measure $\rho^{(1)}\omega_{n}.$
To get a real-valued random variable one fixes a function $u$ on
$X$ and defines the random variable $\mathcal{N}[u]$ by contraction:
\[
\mathcal{N}[u](x_{1},...,x_{N}):=u(x_{1})+....+u(x_{N}),
\]
often called a \emph{linear statistic} in the statistical mechanics
literature. In particular, if $u=1_{E}$ is the characteristic function
of a subset $E$ of $X,$ then $\mathcal{N}[u](x_{1},...,x_{N})$
counts the number of $x_{i}$ contained in $E.$ By \ref{eq:intro conv towards equ}
the expected value of the random measure \ref{eq:intro random measure}
divided by $N$ converges weakly to the \emph{equilibrium measure}
of $(X,\phi).$ In fact, one actually has convergence in \emph{probability,
}i.e. a\emph{ }(weak) ``law of large numbers'':
\begin{thm}
\label{thm:conv in prob}Assume that $(\phi,\mu)$ has the Bernstein-Markov
property and denote by $\mu_{\phi}$ the corresponding equilibrium
measure (supported on the support of $\mu).$ Let $u$ be a bounded
continuous function on $(X,\mu).$ Then 
\begin{equation}
\frac{1}{N_{k}}\mathcal{N}_{k}[u]\rightarrow\int_{X}\mu_{\phi}u\label{eq:conv in prob}
\end{equation}
in probability when $k$ tends to infinity at a rate of order $o(k^{-n}),$
i.e. 
\[
\textrm{Prob}_{k}(\{(x_{1},...,x_{N_{k}}):\,\left|k^{-n}(u(x_{1})+....+u(x_{N_{k}}))-\int_{X}\mu_{\phi}u\right|>\epsilon\})\leq\frac{C}{\epsilon k^{n}}
\]
 for some constant $C$ independent of $\epsilon$ and $k.$ 
\end{thm}
Note that it follows from basic integration theory that the convergence
also holds if $u$ is the characteristic function of a, say smooth,
domain $E$ in $X,$ as long as the limiting equilibrium measure $\mu_{\phi}$
is absolutely continuous (w.r.t. a smooth volume form). In particular,
this happens in the strongly regular case. Theorem \ref{thm:conv in prob}
follows from the convergence of the expectations together with the
following simple variance estimate: 
\[
\textrm{Var\ensuremath{(}}\mathcal{N}_{k}[u]):=\E(\widetilde{\mathcal{N}_{k}}[u])^{2})=O(k^{n})
\]
for any $u$ as above, where $\widetilde{\mathcal{N}_{k}}[u]$ is
the ``fluctuation'' 
\[
\widetilde{\mathcal{N}_{k}}[u]:=\mathcal{N}_{k}[u]-\E(\mathcal{N}_{k}[u])
\]
 of the random variable $\mathcal{N}_{k}[u].$ Before continuing we
point out that by the large deviation results in \cite{berman ldp}
the convergence in the previous theorem in fact holds at the rate
$O(k^{-(n+1)}).$ 

Next, the fluctuations in the bulk are considered for functions $u$
which are Lipschitz continuous, which equivalently means that differential
$du$ is point-wise defined almost everywhere on $X$ and in $L_{loc}^{\infty}.$
In particular, given a continuous Riemannian metric $g$ on a (measurable)
subset $S\subset X$ the Dirichlet norm of $u$ is finite and defined
by 
\[
\left\Vert du\right\Vert _{(S,g)}^{2}:=\int_{S}|du|_{g}^{2}dV_{g},
\]
In the present setting $g$ mainly arises as the Kähler metric in
the bulk of $S$ defined by the Kähler form corresponding to $\phi$
(formula \ref{eq:K=0000E4hler metric in bulk intro}), when $u$ is
supported in the bulk of $S.$ But in fact, the corresponding Dirichlet
norm is defined on $S$ for any Lipschitz continuous function $u$
(see Section \ref{sec:The-pluripotential-equilibrium}). The main
result is the following Central Limit Theorem (CLT), which may be
interpreted as saying that the (scaled) fluctuations of the random
measure \ref{eq:intro random measure} converges in distribution to
the Laplacian of the \emph{Gaussian free field} in the bulk (defined
with respect to the Kähler metric $\omega_{\phi})$ \cite{she}. 
\begin{thm}
\label{thm:conv of laplace}Assume that the weight $\phi$ is in $C_{loc}^{1,1}$
and that the volume form $\omega_{n}$ is continuous. Denote by $S$
the support of the equilibrium measure of $(X,\phi).$ 
\begin{itemize}
\item Assume that $u$ is a Lipschitz function on $X$ supported in a compact
subset of the bulk. Then 
\begin{equation}
\lim_{k\rightarrow\infty}\E(e^{-tk^{-(n-1)/2}\widetilde{\mathcal{N}_{k}}[u],})=\exp(\frac{t^{2}}{8\pi}(\left\Vert du\right\Vert _{(S,\omega_{\phi})}^{2})\label{eq:asympt of Laplace tr in Thm intro}
\end{equation}
in the $C^{\infty}-$topology when $t$ is restricted to a compact
subset of $\C.$ In particular, the variance of $\mathcal{N}[u]$
has the following asymptotics 
\[
\textrm{Var\ensuremath{_{k}(}}\mathcal{N}[u])=\frac{k^{n-1}}{4\pi}(\left\Vert du\right\Vert _{(S,\omega_{\phi})}^{2})+o(k^{n-1})
\]
 and 
\begin{equation}
k^{-(n-1)/2}\widetilde{\mathcal{N}_{k}}[u]:=N^{(1+1/n)/2}\frac{\sum_{i=1}^{N}(u(x_{i})-\E(u(x_{i}))}{N}\label{eq:empirical means in theorem clt into}
\end{equation}
 (where $N=N_{k}\sim k^{n})$ converges in distribution, as $N\rightarrow\infty,$
to a centered normal random variable with mean zero and variance $\frac{1}{4\pi}\left\Vert du\right\Vert _{\omega_{\phi}}^{2}.$
\item For a general continuous function $u$ on $X$ whose differential
$u$ exists almost everywhere the following variance estimate holds:
\[
\frac{k^{n-1}}{4\pi}(\left\Vert du\right\Vert _{(S,\omega_{\phi})}^{2})+o(k^{n-1})\leq\textrm{Var\ensuremath{_{k}(}}\mathcal{N}[u])\leq o(k^{n}),
\]

\end{itemize}
\end{thm}
Let us make some remarks:
\begin{itemize}
\item The assumptions on $\phi$ and $u$ appear to be essentially sharp,
in general (as discussed in Section \ref{sub:Relations-to-recent}).
\item The scaling by $N^{(1+1/n)/2}$ in formula \ref{eq:empirical means in theorem clt into}
gives a gain by a factor $N^{1/2n}$ compared to the classical case
of the CLT for sample averages of independent random variables (appearing
when the points $x_{i}$ are independent and identically distributed).
As explained in Section \ref{sec:Outlook-on-relations} the Large
Deviation Principle established in \cite{berman ldp} provides a simple
heuristic explanation for the scaling above and for the asymptotics
of the variance. 
\item The special case $n=1,$ i.e. when $X$ is a Riemann surface, is singled
out by the fact that the variance of $\mathcal{N}[u]$ is bounded
(i.e. no scaling is required) and its leading asymptotics are independent
of the weight $\phi,$ as follows from the conformal invariance of
the Dirichlet norm when $n=1.$ 
\item Due to the presence of second order phase transitions (when the weight
$\phi$ is perturbed), a central limit theorem for general smooth
functions $u$ - not supported in the bulk - is not to be expected
(see the discussion in Section \ref{sub:Relations-to-phase}). 
\end{itemize}
Applying the previous theorem gives the following normalized version
of the CLT (using \cite{s2} when $n>1):$
\begin{cor}
\label{cor:clt}Assume that the weight $\phi$ is in $C_{loc}^{1,1}$
and that the volume form $\omega_{n}$ is continuous. Let $u$ be
a Lipschitz function on $X$ such that $\left\Vert du\right\Vert _{(S,\omega_{\phi})}^{2}\neq0.$
When $n=1$ assume moreover that $u$ is supported in a compact subset
of the bulk. Then the normalized random variable $\widetilde{\mathcal{N}_{k}}[u]/\sqrt{\textrm{Var\ensuremath{(}}\mathcal{N}_{k}[u])}$
converges in distribution to the standard normal variable with mean
zero and unit variance.
\end{cor}
Just like Theorem \ref{thm:intro bulk univ} the previous results
may be interpreted as a universality result (compare the discussion
in \cite{dei1}). The condition that $\left\Vert du\right\Vert _{(S,\omega_{\phi})}^{2}\neq0$
is natural since the CLT does not hold if $u$ is a constant function
(indeed, the variance then vanishes for any $k).$ The validity of
the normalized CLT when $n>1$ should be contrasted with the failure
of the normalized CLT in the ``real setting'' when $n=1$ (see Section
\ref{sub:Relation-to-previous}).
\begin{rem}
The previous results are actually shown to hold in a more general
setting where $(kL,k\phi)$ is replaced by $(kL+F,k\phi+\phi_{F})$
were $(F,\phi_{F})$ is a Hermitian holomorphic line bundle with suitable
regularity properties. In fact, this flexibility will allow us to
pass directly from variance asymptotics to a central limit theorem.
\end{rem}

\subsection{\label{sub:Relation-to-previous}Relation to previous results}

The main point of the present paper is to apply techniques from complex
geometry/pluripotential theory, in particular $\overline{\partial}$-estimates,
to determinantal point processes. It should be emphasized that in
the case of a smooth weight $\phi$ corresponding to a\emph{ smooth
positively curved} metric on $L$ the asymptotic results on the corresponding
Bergman kernels are well-known and go back to the work of Tian, Bouche,
Zelditch, Catlin and others. For the decay estimate in Theorem \ref{thm:intro decay}
in a $\C^{n}$-setting see \cite{del,li}. Note that by an example
of M.Christ the rate of decay in Theorem \ref{thm:intro decay} is
essentially optimal. The extension to smooth non-positively curved
metrics and the relation to equilibrium measures was initiated in
\cite{berm45,berm4} and then developed to less regular weights and
measures in \cite{b-b,b-b-w}. In the smooth positively curved case
Bergman kernel asymptotics have already been applied and developed
extensively by Shiffman-Zelditch and their collaborators in the different
context of random zeroes of holomorphic sections (defined with respect
to the Gaussian probability measure on the Hilbert space $\mathcal{H}(X,kL)$).
For example, universality of the corresponding correlation functions
was proved in \cite{b-s-z} and a central limit theorem (when $n=1)$
was obtained in \cite{sz4}. 

Let us next compare the results in the present paper with the results
in the extensively studied one-dimensional ``real setting'' appearing
when the reference measure $\mu$ is the Euclidean measure on $\R.$
The corresponding determinantal random point process then coincides
with the Hermitian random matrix model, with the points $x_{i}$ representing
the eigenvalues of the corresponding random matrices. In this setting
the corresponding bulk universality holds at length scales of the
order $k^{-1}$ and the limiting kernel is then the sine kernel (the
bulk is then usually defined as the maximal open set in $\R$ where
the corresponding equilibrium measure has a positive continuous density;
see \cite{pa2} where mean-field theory methods are used and \cite{d-}
for the real-analytic case, where Riemann-Hilbert methods are used).
For the convergence in probability, towards the equilibrium measure
(which is a special case of Theorem \ref{thm:conv in prob}) see \cite{pa1}
and references therein. The analog in the one-dimensional real setting
of the CLT in Theorem \ref{thm:conv of laplace} was obtained in the
seminal work \cite{j} for a sufficiently smooth $u$ and under the
assumption that the weight $\phi$ be sufficiently smooth and that
the support $S\subset\R$ of the corresponding equilibrium measure
be connected (which is the case when, for example, $\phi(x)$ is strictly
convex on $\R).$ The limiting variance is then given by a Sobolev
$1/2-$type norm. The proof in \cite{j} used the method of Ward identities
originating in Quantum Field Theory to compute the second order asymptotics
of the corresponding Laplace transform (appearing in formula \ref{eq:asympt of Laplace tr in Thm intro}).
The latter asymptotics is an analog of the classical Strong Szegö
limit theorem for Toeplitz determinants (concerning the case when
$\mu$ is the invariant measure on $S^{1})$. Interestingly, as shown
in \cite{pa1b} in the case when the support $S\subset\R$ has several
components the CLT does not hold in general (a counter-example is
obtained in \cite{pa1b} for a non-convex real analytic $\phi$ with
$u$ linear on the support). More precisely, as shown in \cite{pa1b}
the corresponding variance is bounded, but not convergent (it is asymptotically
periodic in $N$ as indicated by the formal argument in \cite{b-d-e})
and even the normalized version of the CLT in Corollary \ref{cor:clt}
fails.

In the present complex setting, in the special case when $X=\C$ (and
$\phi(z)$ has super logarithmic growth), Theorem \ref{thm:conv of laplace}
was obtained, independently, in \cite{a-h-m1} for real-analytic $\phi$
and smooth $u.$ The proof in \cite{a-h-m1} uses the method of cumulants,
which is related to the combinatorial approach for central limit theorems
for general determinantal point processes used in \cite{s2} (where
certain estimates on the variance are assumed, as recalled in the
proof of Theorem\ref{cor:clt}). Just as in the present paper, the
key analytic input in \cite{a-h-m1} is Bergman kernel asymptotics,
obtained using the method introduced in \cite{berm4} (see \cite{a-h-m2}).
For the special case where $\phi=\left|z\right|^{2}$ in $\C$ a more
general form of Theorem \ref{thm:conv of laplace} was obtained in
\cite{r-v0} for any $u$ which is $C^{1}-$smooth, using combinatorics
of cumulants. In particular, it is not assumed in \cite{r-v0} that
$u$ be supported in the bulk, which leads to a boundary contribution
in the formula for the limiting variance.

\subsection{\label{sub:Relations-to-recent}Relations to recent developments
and outlook}

The original version of the present paper appeared as a preprint on
ArXiv in 2008 (which also contained some results on links to asymptotics
of direct image bundles that have been removed as they appear in \cite{berman4 och halv}).
Since then there has been various new developments, as will be briefly
recalled next. A central limit theorem allowing general (smooth and
bounded) $u$ in the one-dimensional case of the complex plane was
established in \cite{a-h-m3} using the method of Ward identities
(see Remark \ref{rem:ward}). It was assumed that $\phi$ be real
analytic and the boundary $S$ be a connected domain with real analytic
boundary and that $\Delta\phi>0$ in a neighborhood of $S.$ The corresponding
limiting variance can then by expressed as the Dirichlet norm of the
harmonic extension of $u$ from $S$ to all of $\C,$ which amounts
to adding a boundary contribution to the Dirichlet norm (as in \cite{r-v2}).
As pointed out in Section \ref{sec:Outlook-on-relations} this can
- at a heuristic level - be explained in terms of the general Large
Deviation Principle in \cite{berman ldp} and related to the absence
of second order phase transitions. Very recently, the results in \cite{a-h-m3}
concerning $X=\C$ have been generalized to less regular data $\phi$
\cite{l-s,b-b-n-y} (with $u$ assumed almost $C^{4}-$smooth; see
Section \ref{sub:Relations-to-phase}). As for the scaling limits
of the correlation function at the boundary/edge of the support they
were established in \cite{a-k-m} under suitable regularity and symmetry
assumptions. It would be very interesting to consider the behavior
at the boundary in higher dimensions. This appears to be a very challenging
problem as it seems hard to say anything useful about the boundary
regularity of the support $S$ of the equilibrium measure, in general.
In the presence of toric and circular symmetry results in this direction
have been obtained recently in \cite{p-s,r-s,z-z}. 

In another direction it was shown in \cite{berman4 och halv} that
a sharp version of the Central Limit Theorem in Theorem \ref{thm:conv of laplace}
holds on any Riemann surface when $dd^{c}\phi$ is a Kähler metric
with constant curvature. The sharpness means that the convergence
of the Laplace transforms of the corresponding laws (formula \ref{eq:asympt of Laplace tr in Thm intro})
hold for any test function $u$ with finite Dirichlet norm, $\left\Vert du\right\Vert ^{2}<\infty$
(in the case of the Riemann sphere the convergence in distribution
of the laws was first shown in \cite{r-v2}). However, as pointed
out in \cite{berman4 och halv}, the corresponding statement fails
in higher dimensions (for any given $\phi).$ The point is that when
$n>1,$ even if $\left\Vert du\right\Vert ^{2}$ is assumed finite
the local integrals of $e^{-u}$ may, in general, diverge and hence
the Laplace transform appearing in the left hand side of formula \ref{eq:asympt of Laplace tr in Thm intro},
may diverge. From this point of view the assumption that $u$ be Lipschitz
used in the present paper appears to be essentially optimal. 

Let us also mention the recent work \cite{b-h} where determinantal
point processes defined by real multivariate orthogonal polynomials
are applied to numerical integration, using a Monte Carlo type approach.
In particular, a CLT (analogous to Theorem \ref{thm:conv of laplace})
is established in the ``real setting'' of a measure $\mu$ supported
on the unit-cube in $\R^{n}$ with $u$ a $C^{1}-$smooth function
(supported in the interior of the unit-cube). In the light of \cite{b-h}
the present results in particular provide a theoretical base for numerical
integration of functions $u$ which are periodic in $\R^{2n}$ (by
identifying the fundamental domain with the Abelian variety $X:=\C^{n}+i\C^{n})/\Lambda,$
for $\Lambda=\Z^{n}+i\Z^{n}).$ But we shall not go further into this
here.

It would also be interesting to study universality properties for
general ``beta deformations'' of the determinantal point processes
considered here. Such random point processes are obtained by raising
the Slater determinant appearing in formula \ref{eq:slater det intro}
to the $\beta$th power, for a given real number $\beta$ (by \cite{berman ldp}
the empirical measure still converge in probability towards the equilibrium
measure in the many particle limit). In one complex dimension such
powers were introduced by Laughlin \cite{lau} to explain the experimentally
observed \emph{fractional} Quantum Hall Effect (where the fraction
in question appears as $1/\beta$ when $\beta$ is a suitable positive
integer). For very recent field theoretical works on the Quantum Hall
Effect on Riemann surfaces see the survey \cite{kle} and references
therein. In another direction it was shown in\cite{berm5} that letting
$\beta$ depend on $k,$ yields a probabilistic construction of Kähler-Einstein
metrics $\omega_{KE}$ on complex algebraic varieties $X.$ More precisely,
this happens when $\beta=\pm1/k,$ where the sign is the opposite
sign of the Ricci curvature of $\omega_{KE}.$ In statistical mechanical
terms this corresponds to looking at a limit of fixed non-zero temperature,
which brings entropy into the picture. It would be very interesting
to understand the connections between the latter probabilistic approach
to Kähler-Einstein metrics, using canonical random point processes
and the program of Ferrari-Klevtsov-Zelditch \cite{fkz}, which is
based on random Bergman metrics, i.e. probability measures on the
symmetric spaces $GL(N,\C)/U(N)$ rather than on the $N$ fold symmetric
products of $X.$

\subsection*{Acknowledgments}

It is a pleasure to thank Sébastien Boucksom, David Witt-Nyström,
Frédéric Faure and Jeff Steif for stimulating and illuminating discussions.
The author is particularly grateful to Bo Berndtsson for helpful discussions
concerning Theorem \ref{thm:agmond}. Thanks also to the referee for
comments that helped to improve the exposition.

\subsection*{Organization}

After having introduced the notation and general setup below we illustrate
in Section \ref{sec:Examples} the general geometric setup in the
special case when $X$ is complex projective space, explaining the
relations to orthogonal polynomials and Coulomb and fermion gases.
Then, in Section \ref{sec:The-pluripotential-equilibrium}, we recall
the definition of the pluripotential equilibrium measure and define
its (weak) bulk. In Section \ref{sec:Weighted-estimates-for} we provide
weighted $L^{2}-$estimates for $\bar{\partial}$ formulated in terms
of the equilibrium potential. The latter estimates are then applied
in Section \ref{sec:Asymptotics-for-Bergman} to obtain asymptotics
for Bergman kernels and correlations (proving in particular Theorems
\ref{thm:intro bulk univ}, \ref{thm:intro decay}). In Section \ref{sec:Asymptotics-for-linear}
the main results concerning asymptotics of linear statitistics are
proved, using the asymptotics in Section \ref{sec:Asymptotics-for-Bergman}.
An alternative proof of the CLT using second order expansions is also
given, for smooth data. In the final section an outlook on the relations
between the CLT in Theorem \ref{thm:conv of laplace}, the Large Deviation
Principle (LDP) in \cite{berman ldp} and phase transitions is given.
This leads to a suggestive picture for a general CLT taking boundary
contributions into account, which is consistent with the one-dimensional
results in \cite{r-v2,a-h-m3,l-s,b-b-n-y}.

\subsection{\label{sub:Notation-and-general}Notation and general setup}

\subsubsection*{Weights on line bundles\protect\footnote{general references for this section are the books \cite{gr-ha,de4}.}}

Let $L$ be a holomorphic line bundle over a compact complex manifold
$X.$ We will represent an Hermitian metric on $L$ by its \emph{weight}
$\phi.$ In practice, $\phi$ may be defined as certain collection
of \emph{local} functions. Namely, let $s^{U}$ be a local holomorphic
trivializing section of $L$ over an open set $U$ (i.e. $s^{U}(x)\neq0$
for $x$ in $U)$. Then locally, $\left|s^{U}(z)\right|_{\phi}^{2}=:e^{-\phi^{U}(z)}.$
If $\alpha$ is a holomorphic \emph{section} with values in $L,$
then over $U$ it may be locally written as $\alpha=f^{U}\cdot s^{U},$
where $f^{U}$ is a local holomorphic \emph{function.} In order to
simplify the notation we will usually omit the dependence on the set
$U$ and $s^{U}$ and simply say that $f$ is a local holomorphic
function representing the section $\alpha.$ The point-wise norm of
$\alpha$ may then be locally expressed as
\begin{equation}
\left|\alpha\right|_{\phi}^{2}=\left|f\right|^{2}e^{-\phi},\label{eq:ptwise norm}
\end{equation}
but it should be emphasized that it defines a \emph{global} function
on $X.$

The canonical curvature two-form of $L$ is the global form on $X,$
locally expressed as $\partial\overline{\partial}\phi$ and the normalized
curvature form 
\[
\omega_{\phi}:=i\partial\overline{\partial}\phi/2\pi=:dd^{c}\phi
\]
 (where $d^{c}:=i(-\partial+\overline{\partial})/4\pi)$ represents
the first Chern class $c_{1}(L)$ of $L$ in the second real de Rham
cohomology group of $X.$ The curvature form of a smooth weight is
said to be \emph{positive} at the point $x$ if the local Hermitian
matrix $(\frac{\partial^{2}\phi}{\partial z_{i}\partial\bar{z_{j}}})$
is positive definite at the point $x$ (i.e. $dd^{c}\phi_{x}>0).$
This means that the curvature is positive when $\phi(z)$ is strictly
\emph{plurisubharmonic (spsh)} i.e. strictly subharmonic along local
complex lines. In differential geometric terms this means that the
two-form $\omega_{\phi}$ defines a \emph{Kähler metric,} i.e. the
corresponding symmetric two-tensor $\omega_{\phi}(\cdot,J\cdot)$
is a Riemannian metric compatible with the complex structure $J$
on $X.$ A line bundle is said to be ample (or positive) if admits
a smooth metric with positive curvature. More generally, a weight
$\psi$ on $L$ is called (possibly) \emph{singular} if $\left|\psi\right|$
is locally integrable. Then the curvature is well-defined as a $(1,1)-$current
on $X.$ The curvature current of a singular metric is called \emph{positive}
if $\psi$ may be locally represented by a plurisubharmonic function
and $\psi$ will then simply be called a \emph{psh weight. }A line
bundle $L$ is \emph{big} if admits a psh weigh $\psi$ whose curvature
current is bounded from below by a Kähler form.

Further fixing an Hermitian metric two-form $\omega$ on $X$ with
associated volume form $\omega_{n}$ gives a pair $(\phi,\omega_{n})$
that will be called a weighted measure. It induces an inner product
on the space $H^{0}(X,L)$ of holomorphic global sections of $L$
by declaring 

\begin{equation}
\left\Vert \alpha\right\Vert _{\phi}^{2}:=\int_{X}\left|\alpha\right|_{\phi}^{2}\omega_{n},\label{eq:norm restr}
\end{equation}
The corresponding Hilbert space will be denoted by $\mathcal{H}(X,L)$
and its Bergman kernel by $K(x,y),$ which is a section of the pulled
back line bundle $L\boxtimes\overline{L}$ over $X\times\overline{X}$
(see section \ref{sec:Asymptotics-for-Bergman}). 

The Hermitian line bundle $(L,\phi)$ over $X$ induces, in functorial
way, Hermitian line bundles over all products of $X$ (and its conjugate
$\overline{X}$) and we will usually keep the notation $\phi$ for
the corresponding weights. For example, we will write 
\[
\left|K(x,y)\right|_{\phi}^{2}:=\left|K(z,w)\right|^{2}e^{-\phi(z)}e^{-\phi(w)}
\]
where the right hand side is strictly speaking only defined when both
$x$ and $y$ are contained in an open set $U$ where $L$ has been
trivialized as above. When studying asymptotics we will replace $L$
by its $k$ th tensor power, written as $kL$ in additive notation.
The induced weight on $kL$ may then be written as $k\phi.$ A subindex
$k$ will indicate that the object is defined w.r.t the weight. $k\phi$
on $kL$ for $\phi$ a fixed weight on $L.$

\subsubsection*{Regularity assumptions}

A weighted measure $(\phi,\mu)$ will be called \emph{strongly regular}
if the weight $\phi$ is locally $\mathcal{C}^{1,1}$-smooth (i.e.
it is differentiable and all of its first partial derivatives are
locally Lipschitz continuous) and $\mu=\omega_{n}$ is the volume
form of a continuous metric $\omega$ on $X.$ Moreover, if $(X,L)=(\P^{n},\mathcal{O}(1)),$
where is $\P^{n}$ the complex projective space, viewed as a compactification
of its affine piece $\C^{n},$ then we also allow $\omega_{n}$ to
be defined by the Lebesgue measure on $\C^{n}$ as long as the corresponding
weight \emph{function} $\phi(z)$ on $\C^{n}$ has super logarithmic
growth (formula \ref{eq:ass on growth of phi} below) with $\phi\in\mathcal{C}_{loc}^{1,1}(\C^{n}).$

\subsubsection*{Probability notation}

Given a probability space $(Y,\gamma),$ i.e. a measure space where
$\gamma(X)=1,$ a measurable function $\mathcal{N}$ on $(Y,\gamma)$
is called a \emph{random variable.} Its integral w.r.t to $Y$ is
denoted by $\E(\mathcal{N})$ and called the \emph{expectation} of
$\mathcal{N}.$ Recall also that if $\mathcal{N}$ takes values in
a space $Z$ then the pushforward of $\gamma$ under $\mathcal{N}$
is called the\emph{ law of $\mathcal{N}$ on $Z.$} A subindex $k$
will indicate that the object is defined w.r.t. the probability measure
on $Y=X^{N_{k}},$ defined by the density \ref{eq:intro prob measure}
induced by a weighted measure $(\phi,\mu).$ 

Occasionally, we will also consider the probability measures defined
by the Bergman kernels $K_{k\phi+\phi_{F}}$ associated to a sequence
of Hermitian line bundles $(kL+F,k\phi+\phi_{F})$ (and a fixed reference
measure $\mu)$ and we will then write $\E=\E_{k\phi+\phi_{F}}$ etc.

\section{\label{sec:Examples}Examples }

In this section we will illustrate our setup in the concrete case
when $X$ is the complex projective space. But it may also be worth
pointing out that another concrete setting appears when $X:=\C^{n}/\Lambda$
is a principally polarized torus (Abelian variety), in which case
$H^{0}(X,kL)$ may be identified with the space of theta functions
on $\C^{n}$ at level $k,$ which are $\Lambda-$quasi periodic. In
particular, the latter setting gives a geometric approach to the one-dimensional
setting in \cite{f1}.

\subsection{From projective space to orthogonal polynomials and Vandermonde determinants}

It is a classical fact that $\C^{n}$ is compactified by the complex
projective space $X:=\P^{n}.$ Let $L$ be the hyperplane line bundle
$\mathcal{O}(1)$ on $\P^{n}.$ Then $H^{0}(X,kL)$ is the space of
all complex homogeneous polynomials of total degree $k$ in $\C^{n+1},$
which is isomorphic to the vector space $\mathcal{H}_{k}(\C^{n})$
of all polynomials in $\C^{n}$ of total degree at most $k.$ Indeed,
fix a global holomorphic section $s$ of $\mathcal{O}(1),$ whose
zero-set is $\P^{n}-\C^{n},$ the'' hyper plane at infinity''. Then
any section $s_{k}$ of $L^{\otimes k}$ over the open subset $U:=\C^{n}$
may be written as 
\[
s_{k}(z)=p_{k}s^{\otimes k}
\]
 where $p_{k}$ is in $\mathcal{H}_{k}(\C^{n})$ (concretely, this
amounts to ``dehomogenizing'' $s_{k}$). Moreover, the point-wise
norms with respect to a metric on $k\mathcal{O}(1)$ induced by a
given locally bounded metric $h$ on $\mathcal{O}(1)$ become 
\begin{equation}
\left|s_{k}(z)\right|_{h^{\otimes k}}^{2}=\left|p_{k}(z)\right|^{2}e^{-k\phi(z)}\label{eq:ptwise norms on pn}
\end{equation}
for some function $\phi(z)$ on $\C^{n},$ that we will call the \emph{weight
function.} As is well-known, this gives a correspondence between locally
bounded metrics $h$ on $\mathcal{O}(1)$ and weight functions $\phi(z)$
of the form 

\begin{equation}
\phi(z)=\phi_{FS}(z)+u(z):=\ln(1+\left|z\right|^{2})+u(z),\label{eq:phi with log growth}
\end{equation}
 where $u$ is a locally bounded function on $\C^{n}.$ In particular,
a subclass of weights corresponding to \emph{smooth} metrics on $\mathcal{O}(1)$
are obtained by taking $u\in\mathcal{C}_{c}^{\infty}(\C^{n}).$ Note
that the metric $h_{FS}$ corresponding to $\phi_{FS}(z)$ is the
Fubini-Study metric on $\mathcal{O}(1)$ which is characterized (up
to a constant) by its invariance under the $SU(n)-$action. Its (normalized)
curvature form $\omega_{FS}:=dd^{c}\phi_{FS}$ is the called the Fubini-Study
metric on $\P^{n}$ and a simple calculation shows that the corresponding
volume form is given by
\[
(\omega_{FS})_{n}:=(dd^{c}\phi_{FS})^{n}/n!=e^{-(n+1)\phi_{FS}}(\frac{i}{2})^{n}dz\wedge d\bar{z}
\]
where $(\frac{i}{2})^{n}dz\wedge d\bar{z}$ denotes the Lebesgue measure
on $\C^{n}.$ The global norm of $s_{k}$ induced by the weighted
measure $(\phi,(\omega_{FS})_{n})$ may hence be represented as 

\begin{equation}
\left\Vert s_{k}\right\Vert _{(\phi,\omega_{FS})}^{2}:=\int_{\C^{n}}\left|p_{k}(z)\right|^{2}e^{-k\phi(z)}(\omega_{FS})_{n}.\label{eq:pk norm on pn}
\end{equation}
Alternatively, the weight $\phi$ itself induces a measure $e^{-(n+1)\phi(z)}(\frac{i}{2})^{n}dz\wedge d\bar{z}.$
The corresponding norm is hence given by 
\[
\left\Vert s_{k}\right\Vert _{\phi}^{2}:=\int_{\C^{n}}\left|p_{k}(z)\right|^{2}e^{-(k+n+1)(\phi(z))}(\frac{i}{2})^{n}dz\wedge d\bar{z}
\]
Note that the the contribution from the factor $e^{-(n+1)\phi}$ makes
sure that the integrals are finite. 

The corresponding determinantal probability density \ref{eq:slater det}
may in this case be expressed explicitly as 
\begin{equation}
\frac{1}{\mathcal{Z}_{k\phi}}|\Delta^{(N_{k})}(z_{1},...,z_{N_{k}})|^{2}e^{-k\phi(z_{1})}\cdots e^{-k\phi(z_{N_{k}})},\label{eq:vanderm density}
\end{equation}
where $\Delta^{(N_{k})}(z_{1},...,z_{N_{k}})$ is the higher dimensional
\emph{Vandermonde determinant,} i.e. the Slater determinant $\det S$
corresponding to a bases $S$ of multinomials and where $\mathcal{Z}_{k\phi}$
is the corresponding normalizing factor (compare Lemma \ref{lem:amazing id}).

\subsubsection{The setting of super logarithmic growth and sections vanishing along
a hypersurface}

A variant of the previous setting arises if one insists on using the
Lebesgue measure as the integration measure defining the norms in
\ref{eq:pk norm on pn}. Then $\phi(z)$ has to have slightly larger
growth than in formula \ref{eq:phi with log growth} in order to get
finite norms. More precisely, we then assume that $\phi$ has super
logarithmic growth in the sense that 
\begin{equation}
\phi(z)\geq(1+\epsilon)\ln\left|z\right|^{2},\,\,\textrm{when\,}\left|z\right|>>1\label{eq:ass on growth of phi}
\end{equation}
 for some positive number $\epsilon.$ It should be emphasized that
such a weight $\phi$ does not correspond to a locally bounded metric
$h$ on $\mathcal{O}(1).$ But as shown in \cite{berm4} a slight
modification of the arguments apply to this super logarithmic setting,
as well. The key point is that the growth condition \ref{eq:ass on growth of phi}
forces the corresponding equilibrium measure to be compactly supported
in $\C^{n}.$ The model case is when $\phi(z)=\left|z\right|^{2}.$
Then the equilibrium measure is (up to a multiplicative constant)
the Lebesgue measure on the unit ball. 
\begin{rem}
\label{rem:sections vanishing along}Another variant of the geometric
setting of a line bundle $L\rightarrow X$ endowed with a, say smooth,
weight $\phi$ is obtained by fixing a smooth complex hypersurface
$Z$ in $X$ (of codimension one). Let $H_{k\lambda Z}$ be the subspace
of $H^{0}(X,kL)$ consisting of all sections vanishing to order $[k\lambda]$
along $Z$ for a fixed sufficiently small positive number $\lambda.$
Then any continuous Hermitian metric $\left\Vert \cdot\right\Vert $
(with curvature form $\omega)$ and a volume form $\omega_{n}$ on
$X$ induce by restriction, an inner product on the subspace $H_{k\lambda Z}.$
Hence, we can associate a sequence of determinantal point-processes
to the corresponding sequence of Hilbert spaces $H_{k\lambda Z}.$
As shown in \cite[Section 5.5]{berman ldp} the laws of the corresponding
sequence of empirical measures satisfy a large deviation principle
(LDP). The results in the present paper also extends with simple modifications
to the determinantal point processes associated to $H_{k\lambda Z}$
(by replacing the equilibrium potential $\phi_{e}$ used in the present
paper with the corresponding equilibrium potential relative to $\lambda Z,$
obtained by imposing that $\psi$ in formula \ref{eq:extem metric}
has a Lelong number of at least $\lambda$ along $Z$). In fact, the
setting of super logarithmic growth in $\C^{n}$ can be fitted into
this setting in the case when $\phi$ is of the special form 
\begin{equation}
\phi(z)=(1+\epsilon)\log(1+|z|^{2})+u(z),\label{eq:weight to which WR applies}
\end{equation}
 where $u(z)$ extends smoothly from $\C^{n}$ to $\P^{n}.$ Indeed,
one then let $Z$ be a hyperplane in $X:=\P^{n}$ and identifies $\C^{n}$
with $X-Z,$ in the usual way.
\end{rem}

\subsection{\label{sub:The-higher-dimensional}A higher dimensional Coulomb type
gas}

Continuing with the setting of multivariate orthogonal polynomials
in $\C^{n}$ and introducing the Hamiltonian

\[
E_{k\phi}(z_{1},...,z_{N}):=E_{k}(z_{1},...,z_{N_{k}})+k\phi(z_{1})/2+...+k\phi(z_{N_{k}})/2,
\]
where 
\[
E_{k}(z_{1},...,z_{N_{k}})=-\log\left|\Delta^{(N_{k})}(z_{1},...,z_{N_{k}})\right|,
\]
the corresponding probability density \ref{eq:vanderm density} may
be written as a \emph{Boltzmann-Gibbs density }at inverse temperature
$\beta=2$ (in suitable units):
\begin{equation}
\frac{e^{-\beta E_{k}(z_{1},...,z_{N})}}{\mathcal{Z}_{k\phi}},\label{eq:coulomb gas dens}
\end{equation}
describing an ensemble of $N_{k}$ identical particles in thermal
equilibrium interacting by the internal energy $E_{k}(z_{1},...,z_{N})$
and subject to the exterior potential $k\phi/2.$ In particular, in
the one-dimensional case, expanding the Vandermonde determinant reveals
that $E_{k}(z_{1},...,z_{N})$ is precisely the Coulomb interaction
for $N_{k}$ unit-charge particles: 
\[
E_{k}(z_{1},...,z_{N_{k}})=-\frac{1}{2}\sum_{1\leq i,j\leq N}\log|z_{i}-z_{j}|^{2}
\]
(such a gas is also called a one component plasma in the physics literature).
Using mean field theory heuristics one would expect that the corresponding
random point processes satisfy a\emph{ Large Deviation Principle}
(LDP) with a rate function $E(\mu)+\int\phi\mu$ defined on the space
of all probability measures on $\C^{n}$ and with speed $kN,$ i.e.
that 

\[
\mbox{Prob }\left\{ \frac{1}{N}\sum\delta_{z_{i}}\cong\mu\right\} \sim e^{-kN\left(E(\mu)+\int\phi\mu\right)}/Z
\]
holds in the sense of large deviations. As shown in the companion
paper \cite{berman ldp} this is indeed the case (see Section \ref{sec:Outlook-on-relations})
and, in physical terms, it can be interpreted as a higher dimensional
effective fermion-boson correspondence. This LDP is also closely related
to the fact that the corresponding equilibrium measure $MA(\phi_{e})$
(which in the present paper is defined directly in terms of pluripotential
theory in Section \ref{sec:The-pluripotential-equilibrium}) may be
alternatively obtained as the unique minimizer of the total ``macroscopic''
energy $E(\mu)+\int\phi\mu$ appearing as the rate functional above;
see \cite{berman ldp} and reference therein.

\subsection{\label{sub:Normal-matrices}Random normal matrices}

Consider the set of all normal matrices $\mathcal{M}_{N}:=\{M\in gl(N,\C):\,[M,M^{*}]=0\}$
as a Riemannian subvariety of the space $gl(N,\C)$ of all complex
matrices of rank $N$ equipped with the Euclidean metric. A given
weight function $\phi$ of super logarithmic growth induces the following
probability measure on $\mathcal{M}_{N}$ 
\begin{equation}
e^{-N\textrm{Tr}(\phi(M))}dV_{\mathcal{M}_{N}}/\mathcal{Z}{}_{N\phi}\label{eq:prob measure for random matr}
\end{equation}
 where $dV_{\mathcal{M}_{N}}$ is the Riemannian volume measure of
$\mathcal{M}_{N}$ and $\mathcal{Z}{}_{N\phi}$ is a normalizing constant
(usually called the partition function of the corresponding matrix
model \cite{za}). Under the map which associates the (ordered) eigenvalues
$(z_{1},...,z_{N})$ to a matrix $M$ the probability measure \ref{eq:prob measure for random matr}
is pushed forward to a probability measure on $\C^{N}$ which turns
out to coincide with the determinantal probability measure for polynomials
of degree $N-1$ weighted by $\phi$ (when $n=1)$. The corresponding
correlation functions $\rho_{k}^{(m)}$ are hence usually called \emph{eigenvalue
correlation functions} in this context. It should also be pointed
out that the correlation functions corresponding to the weighted set
$(\phi,\mu)$ where $\mu$ is the invariant measure supported on $\R$
(or the unit-circle $T)$ coincide with eigenvalue correlation functions
for random\emph{ Hermitian} (or unitary) matrices, weighted by $\phi,$
which have been extensively studied (cf. \cite{dei2,j,pa2} and references
there in).

\subsection{Free fermions in a magnetic field}

When $n=1$ the weighted polynomials $\Psi_{+,m}:=z^{m}e^{-k\phi(z)/2}$
where $m=0,...,k$ each represent the quantum state of a single spin
$1/2$ quantum particle (=fermion) confined to a plane subject to
a magnetic field $B$ perpendicular to the plane, where the value
of $B$ at the point $z$ is $\frac{i}{2\pi}k$$\frac{\partial^{2}\phi(z)}{\partial z\partial\bar{z}}$
in suitable units (and similarly in higher dimensions; see \cite{sh,berman ldp}
and references therein). Moreover, the states form a linearly independent
set in the lowest possible energy level (i.e. the ground state). More
precisely, this latter fact means that $\Psi_{+,m}$ is an eigenvector
of finite norm with eigenvalue $0$ of the Pauli operator, which in
complex notation may be written as 
\[
(\overline{\partial}{}_{k\phi}+\overline{\partial}{}_{k\phi}^{*})^{2}\Psi_{+,m}=0,
\]
 where $\overline{\partial}{}_{k\phi}$ intertwines the space $S_{+}:=\Omega^{0,0}(\C)$
of spin \emph{up} and the space $S_{-}:=\Omega^{0,1}(\C)$ of spin
\emph{down} particles
\[
\overline{\partial}{}_{k\phi}=\overline{\partial}+\frac{k}{2}\overline{\partial}\phi\wedge:\,\,\,S_{+}\rightarrow S_{-}
\]
 and $\overline{\partial}{}_{k\phi}^{*}$ is its formal adjoint. This
means that the corresponding real ``vector potential'' (i.e. $U(1)-$gauge
field) for the magnetic two-form is given by $k$ times
\[
A:=\frac{1}{2}(\overline{\partial}\phi-\partial\phi)),
\]
where $dA=iB.$ Hence, the particle state $\Psi_{+,m}$ is said to
have spin \emph{up}, since it has no spin down component in $\Omega^{0,1}(\C)$
(defined is the space of element of the form $gd\bar{z}),$ where
$g\in C^{\infty}(\C)).$ The corresponding many particle state of
$N$ free fermions, should, according to the postulates of quantum
mechanics for fermions, be anti-symmetric under an exchange of two
single particle states $\Psi_{m}.$ Hence, it is represented by the
(Slater) determinant $\Psi(z_{1},...,x_{N}):=\det(\Psi_{+,i}(z_{j})).$
In particular, the corresponding probability amplitude coincides (after
normalization) with the corresponding determinantal probability measure
(compare Lemma \ref{lem:amazing id}). The correspondence between
the free fermion representation and the Coulomb bas picture above
can, at a heuristic level, be explained by the process of bosonization
(see \cite{b-v-,berman ldp}). 
\begin{rem}
The Pauli operator above is defined as the square of the Dirac operator
$\mathcal{D}_{kA}:=(\overline{\partial}{}_{k\phi}+\overline{\partial}{}_{k\phi}^{*})$
on the space $S:=S_{+}\oplus S_{-}$of complex spinors, endowed with
the $L^{2}-$norm induced by the Euclidean metric on $\C$ (this setup
corresponds to gyromagnetic ratio $g=2;$ see for example \cite{sh}
and \cite[Chaper 5]{c-k-s} for a physics reference). If one instead
uses the metric induced by the curvature form $B$ - assuming that
$B$ is positive - then the square of the corresponding Pauli operator
on may be expressed as 
\begin{equation}
\mathcal{D}_{A}^{2}=(\frac{1}{4}\nabla_{kA}^{*}\nabla_{kA}-k)\oplus(\frac{1}{4}\nabla_{A}\nabla_{A}^{*}+k),\label{eq:Pauli as Landau}
\end{equation}
where the magnetic Schrödinger operator $\nabla_{kA}^{*}\nabla_{kA}$
is the Landau Hamiltonian for a non-spinning particle subject to the
magnetic vector potential $kA$ (in our general setting this corresponds
to taking the measure $\omega_{n}$ to be the one induced by the $dd^{c}\phi).$
From the complex geometric point of view formula \ref{eq:Pauli as Landau}
is a special case of the Bochner-Kodaira-Nakano formula \cite{gr-ha}.
In particular, in the case of constant positive magnetic field, i.e.
$\phi(z)=|z|^{2},$ the Pauli and the Landau operators are essentially
the same (up to an additive constant depending on the spin). 
\end{rem}

\section{\label{sec:The-pluripotential-equilibrium}The pluripotential equilibrium
measure}

In this section we will give the pluripotential construction of the
measure which will arise as the limiting expected distribution of
the empirical measure of the point processes on $X.$ 

Let $L\rightarrow X$ be an ample line bundle over a compact complex
manifold $X.$ Given a weight $\phi$ on $L,$ that we first only
assume is continuous, the corresponding ``equilibrium weight'' $\phi_{e}$
is defined as the envelope 

\begin{equation}
\phi_{e}(x):=\sup\left\{ \psi(x):\,\,\psi\leq\phi\,\,\textrm{on\ensuremath{\,X}}\right\} .\label{eq:extem metric}
\end{equation}
 where the sup is taken over all continuous psh weights $\psi.$ Then
$\phi_{e}$ is also a continuous psh weight on $L$ \cite{g-z} and
we denote by $D$ the corresponding coincidence set: 
\[
D:=\{\phi_{e}=\phi\}\subset X
\]
so that $D=X$ precisely when $\phi$ is a psh weight. \emph{The equilibrium
measure} (associated to the continuous weight $\phi)$ is in general
defined as the \emph{Monge-Ampère measure} $MA(\phi_{e})$ constructed
in the seminal work of Bedford-Taylor in the local setting (see \cite{g-z}
for the global setting). For a \emph{smooth} psh weight $\psi$ this
measure is simply defined by 

\begin{equation}
MA(\psi):=(dd^{c}\psi)^{n}/n!=(\frac{i}{2\pi})^{n}\det(\frac{\partial^{2}\psi}{\partial z_{i}\partial z_{j}})dz_{1}\wedge d\overline{z}_{1}\wedge...dz_{n}\wedge d\overline{z}_{n}\label{eq:def of mu}
\end{equation}
As is well-known the equilibrium measure $\mu_{\phi_{e}}$ is supported
on $D$ (see below). In the case when $\phi$ is smooth (and not merely
continuous) it was shown in \cite{berm45} that $\phi_{e}$ is $\mathcal{C}^{1,1}-$
smooth and in particular the local derivatives $\frac{\partial^{2}\psi}{\partial z_{i}\partial z_{j}}$
exist almost everywhere on $X$ and are locally bounded. We may then
simply\emph{ define }the equilibrium measure in this setting by the
following measure which has an $L_{loc}^{\infty}-$density\emph{ }
\[
\mu_{\phi_{e}}:=\frac{1}{V}MA(\phi_{e}):=\frac{1}{V}(dd^{c}\phi_{e})^{n}/n!
\]
More precisely, the following theorem holds and is the specialization
to ample line bundles of a general result in \cite{berm45} concerning
\emph{big} line bundle (see Theorem 3.4 and Remark 3.6 there). It
shows that if $\phi$ is class $\mathcal{C}^{1,1}$ on $X,$ than
$\phi_{e}$ is also in the class $\mathcal{C}^{1,1}:$ 
\begin{thm}
\label{thm:reg}Suppose that $L$ is an ample line bundle and that
the given metric $\phi$ on $L$ is in the class $\mathcal{C}^{1,1}.$
Then

(a) $\phi_{e}$ is in the class $\mathcal{C}^{1,1}$ on $X.$ 

(b) The Monge-Ampère measure of $\phi_{e}$ on $X$ is absolutely
continuous with respect to any given volume form and coincides with
the corresponding $L_{loc}^{\infty}$ $(n,n)-$form obtained by a
point-wise calculation: 
\begin{equation}
(dd^{c}\phi_{e})^{n}/n!=\det(dd^{c}\phi_{e})\omega_{n}\label{eq:ptwise repr of equil meas}
\end{equation}

(c) the following identity holds almost everywhere on the set $D:=\{\phi_{e}=\phi\}:$
\begin{equation}
\det(dd^{c}\phi_{e})=\det(dd^{c}\phi)\label{eq:monge on D}
\end{equation}
More precisely, it holds for all points where the second order jet
$(\phi_{e}-\phi)^{(2)}$ exists and vanishes and in particular point-wise
on 
\begin{equation}
\{(\phi_{e}-\phi)^{(2)}=0\}\cap\{\det(dd^{c}\phi)>0\}\label{eq:bulk in statement thm reg}
\end{equation}

(d) Hence, the following identity between measures on $X$ holds:
\begin{equation}
n!V\mu_{\phi_{e}}=(dd^{c}\phi_{e})^{n}=1_{D}(dd^{c}\phi)^{n}=1_{D\cap X(0)}(dd^{c}\phi)^{n},\label{eq:d in theorem on equil}
\end{equation}
 where $X(0)=\{dd^{c}\phi>0\}.$
\end{thm}
We define the set 
\[
S:=D\cap X(0)
\]
that we shall call the \emph{support} of the equilibrium measure $\mu_{\phi_{e}},$
in view of formula \ref{eq:d in theorem on equil}. Next, we are going
to define the\emph{ weak bulk} (of the equilibrium measure associated
to $\phi$). It may seem tempting to define it as the interior of
the support $S$ of the equilibrium measure, but the problem is that
there are essentially no general regularity results for $S$ - for
example it is not clear that, in general, $\overline{\mbox{int \ensuremath{(S)}}}=\bar{S}.$
In fact, it even not clear that the interior $\mbox{int}\ensuremath{(S)}$
is non-empty, in general! (see \cite{Sc} for the construction of
examples where the coincidence set $D$ can be extremely irregular,
in the case $n=1$). 
\begin{defn}
The set in formula \ref{eq:bulk in statement thm reg} above is called
the \emph{weak bulk (of $(X,\phi)).$ }When $\phi$ is assumed to
be in $C_{loc}^{2}$ the \emph{bulk (of $(X,\phi))$ }is defined as
the interior of the support $S$ of the equilibrium measure. For a
general $\phi$ in $C_{loc}^{1,1}$ the bulk is defined as the maximal
open subset of the interior of $S$ where $dd^{c}\phi_{e}$ (or equivalently,
$dd^{c}\phi)$ is represented by a continuous and strictly positive
form (i.e. a continuous Kähler metric).
\end{defn}
The definitions are made so that, in the weak bulk, the density of
the equilibrium measure (w.r.t. $\omega_{n})$ exists and is equal
to $\det(dd^{c}\phi)$ and vanishes a.e. on the complement of the
bulk. Moreover, the bulk is always contained in the weak bulk. We
note that for a general Lipschitz continuous function the Dirichlet
norm $\left\Vert du\right\Vert _{(S,\omega_{\phi})}^{2}$ is well-defined.
Indeed, by the previous regularity theorem 
\[
\left\Vert du\right\Vert _{(S,\omega_{\phi})}^{2}=V\int_{X}\left|du\right|_{\omega_{\phi}}^{2}\mu_{\phi}
\]
 which is well-defined since $\omega_{\phi}>0$ almost everywhere
with respect to $\mu_{\phi}.$
\begin{rem}
In the general case when $L$ is big one defines the weak bulk as
above on the augmented base locus of $X$ (also called the Kähler
locus), which is a (Zariski) open subset of $X.$ But for simplicity
we will mainly stick to the case when $L$ is ample. 
\end{rem}

\subsection{Remarks on regularity properties of the support $S$ \label{rem:reg of bulk}}

Even in the classical one-dimensional case where $(X,L)=(\P^{1},\mathcal{O}(1))$
and $\phi$ is smooth, the equilibrium weight may not have second
derivatives at some points. In fact, when $\phi$ is radial this happens
``generically'' \cite{berm45}. More generally, when $(X,L)$ is
a toric or abelian variety and $\phi$ is invariant under the corresponding
torus action the envelope $\phi_{e}$ may be identified with the convexification
of the function $\Phi(x)$ on $\R^{n}$ corresponding to $\phi.$
For a generic such $\Phi$ the corresponding support $S_{\Phi}$ has
been classified in dimension $n\leq3$ as a domain with piece-wise
smooth boundary, with explicit algebraic singularity type. The proof
uses Arnold's catastrophe theory of Lagrangian singularities (motivated
by the adhesion model in cosmology where $S$ arises in the Eulerian
description of the ``cosmic web''; see \cite{bo} and the appendix
in \cite{gu-m-s}). However, in the general complex geometric setting
there are almost no general results concerning the regularity properties
of the support $S.$ It would be interesting to find general conditions
ensuring that $S$ is a topological domain (i.e $\overline{\mbox{int \ensuremath{(S)}}}=\bar{S})$
with some additional regularity properties. Comparing with the extensively
studied Laplacian case appearing when $n=1$ \cite{c-r} suggests
that a minimal requirement in order to have reasonable regularity
properties is the assumption that $dd^{c}\phi>0$ on the coincidence
set $D$ (which then coincides with the the corresponding support
set $S$). For example, in the setting of sections vanishing along
a hypersurface described in Remark \ref{rem:sections vanishing along}
it has recently been shown in \cite{RWN} that the support $S$ of
the corresponding equilibrium measure is a domain with smooth boundary
under the assumption that $dd^{c}\phi>0$ on all of $X$ and $\lambda$
is sufficiently small (in fact, the complement of $S$ is then even
diffeomorphic to a tubular neighborhood of $Z).$ In particular, this
result applies in the setting of logarithmic growth in $\C^{n}$ as
long as the weight $\phi(z)$ is smooth and strictly plurisubharmonic
and the number $\epsilon$ appearing in formula \ref{eq:weight to which WR applies}
is sufficiently small. Anyway, it should be stressed that an important
point in the present paper is to avoid making any detailed regularity
assumptions on the support $S.$

\section{\label{sec:Weighted-estimates-for}Weighted $L^{2}-$estimates for
$\overline{\partial}$}

In this section we will generalize, by refining the results in \cite{berm45},
some well-known estimates for the $\overline{\partial}-$operator
concerning psh weights to more general weights. More precisely, we
will assume that $\phi$ is a locally $\mathcal{C}^{1,1}-$smooth
weight on the line bundle $L$ over $X.$ When $(X,L)=(\P^{n},\mathcal{O}(1))$
we also allow weights corresponding to a weight function $\phi(z)$
in $\C^{n}$ with super logarithmic growth (see section \ref{sec:Examples}).
But for simplicity we do not consider the latter situation in the
proofs. The simple modifications needed follow precisely as in the
appendix in \cite{berm4}.

We will denote by $K_{X}$ the canonical line bundle of $X,$ whose
smooth sections are $(0,n)-$forms on $X.$ A weight $\phi$ on $L$
induces, without choosing a volume form $\omega_{n}$ on $X,$ an
$L^{2}-$norm on sections $u$ of $L+K_{X}$ that we will write as
\[
\left\Vert u\right\Vert _{\phi}^{2}:=\int_{X}\left|u\right|^{2}e^{-\phi}
\]
In the statement of the following theorem, we will use the fact that
$dd^{c}\phi$ defines a positive form with locally bounded coefficients
in the bulk (by the very definition of the bulk). 
\begin{thm}
\label{thm:horm-kod-for phi}Let $L$ be a big line bundle and $\phi$
a $\mathcal{C}^{1,1}-$smooth weight. Then for any $\overline{\partial}-$closed
$(0,1)-$form $g$ with values in $L+K_{X}$ and supported in the
interior of the bulk, there is a smooth section $u$ with values in
$L+K_{X}$ such that 
\begin{equation}
\overline{\partial}u=g\label{eq:d-bar eq in theorem}
\end{equation}
and 
\begin{equation}
\int_{X}\left|u\right|^{2}e^{-\phi}\leq\int_{X}\left|g\right|_{dd^{c}\phi}^{2}e^{-\phi}.\label{eq:thm horm-kod}
\end{equation}
In particular, the previous estimate holds for any $u$ such that
$u$ is orthogonal to $H^{0}(X,L+K_{X})$ (w.r.t the weight $\phi).$ \end{thm}
\begin{proof}
Let $\psi$ denote a general psh weight on $L.$ By theorem 5.1 in
\cite{de1} the theorem holds with $\phi$ replaced by a (possibly
singular) psh weight $\psi$ if $dd^{c}\phi$ is replaced with the
absolutely continuous part $(dd^{c}\psi)_{c}$ of the Lebesgue decomposition
of the positive form $dd^{c}\psi.$ More precisely, 
\begin{equation}
\int_{X}\left|u\right|^{2}e^{-\psi}\leq\int_{X}\left|g\right|_{(dd^{c}\psi)_{c}}^{2}e^{-\psi}\label{eq:pf thm horm-kod: dem est}
\end{equation}
as long as the r.h.s is finite. Now set $\psi=\phi_{e},$ the equilibrium
weight corresponding to $\phi.$ Since $g$ is supposed to be supported
in the bulk, the regularity Theorem \ref{thm:reg}, gives 
\[
\int_{X}\left|g\right|_{(dd^{c}\phi_{e})_{c}}^{2}e^{-\phi_{e}}=\int_{X}\left|g\right|_{(dd^{c}\phi)}^{2}e^{-\phi}
\]
and since $g$ is, in fact, supposed to be supported in the \emph{pseudo-interior}
of the bulk the latter integral is finite. Finally, using that $\phi_{e}\leq\phi$
on \emph{all} of $X$ finishes the proof of the estimate \ref{eq:thm horm-kod}.
The last statement of the theorem now follows since the estimate \ref{eq:thm horm-kod}
in particular holds for the solution which minimizes the corresponding
$L^{2}-$norm.\end{proof}
\begin{rem}
\label{rem:perturbation by bounded}Given a bounded function $f$
on $X$ it follows immediately from the inequality \ref{eq:thm horm-kod}
that 
\[
\int_{X}\left|u\right|^{2}e^{-(\phi+f)}\leq C_{f}\int_{X}\left|g\right|_{dd^{c}\phi}^{2}e^{-(\phi+f)},\,\,\,C_{f}=e^{2\left\Vert f\right\Vert _{L^{\infty}(X)}}
\]
In particular, the previous estimate holds when $u$ is the solution
to the equation \ref{eq:d-bar eq in theorem} which is minimal wrt
the $L^{2}-$norm on $L$ induced by the weight $\phi+f.$ 
\end{rem}
The previous theorem is a generalization to non-psh weights $\phi$
of the fundamental result of Hörmander-Kodaira. In turn, the next
theorem is a generalization to non-psh weights of a refinement of
the Hörmander-Kodaira estimate which goes back to a twisting trick
in the work of Donelly-Fefferman. See \cite{del,li} for an analogous
result concerning psh weights in $\C^{n}$. 
\begin{thm}
\label{thm:agmond}Let $L$ be a big line bundle, $\phi$ a $\mathcal{C}^{1,1}-$smooth
weight on $L$ and $v$ a smooth function on $E$ such that $dv$
is supported in the interior of the bulk of $(X,\phi)$ and 
\[
(i)\,\left|\overline{\partial}v\right|_{dd^{c}\phi}^{2}\leq1/8\,\,\,\,\,(ii)\,dd^{c}v\geq-dd^{c}\phi/2
\]
there. Then 
\begin{equation}
\int_{X}\left|u\right|^{2}e^{-\phi_{e}+v}\leq2\int_{X}\left|\overline{\partial}u\right|_{dd^{c}\phi}^{2}e^{-\phi_{e}+v}\label{eq:d-f estimate in theorem}
\end{equation}
 for any smooth section $u$ of $L+K_{X}$ orthogonal to the space
$H^{0}(L+K_{X}),$ w.r.t the weight $\phi,$ and such that $\overline{\partial}u$
is supported in the interior of the bulk of $(X,\phi).$ Moreover,
given a bounded function $f$ on $X$ the function $v$ above may
be replaced by $v+f$ at the expence of multiplying the right hand
side in the inequality \ref{eq:d-f estimate in theorem} by $C_{f}:=e^{2\left\Vert f\right\Vert _{L^{\infty}(X)}}.$\end{thm}
\begin{proof}
By assumption
\[
\left\langle u,h\right\rangle _{\phi}=0,\,\,\,\forall h\in H^{0}(X,L+K_{X}).
\]
Equivalently, writing $u_{v}:=ue^{v},$
\begin{equation}
\left\langle u_{v},h\right\rangle _{\phi+v}=0,\,\,\,\forall h\in H^{0}(X,L+K_{X}).\label{eq:pf agmond:og}
\end{equation}
By Leibniz rule 
\begin{equation}
\overline{\partial}u_{v}=(\overline{\partial}u+\overline{\partial}vu)e^{v},\label{eq:pf agmond: leibn}
\end{equation}
 which by assumption is supported in the bulk of $(X,\phi).$ Hence,
applying the estimate \ref{eq:pf thm horm-kod: dem est} in the proof
of the previous theorem to $\psi=\phi_{e}+v$ gives, since by assumption
$ii$ $(\phi_{e}+v)$ is a psh weight 
\[
\int_{X}\left|u_{v}\right|^{2}e^{-(\phi_{e}+v)}\leq\int_{X}\left|\overline{\partial}u_{v}\right|_{dd^{c}(\phi_{e}+v)}^{2}e^{-(\phi_{e}+v)}\leq\int_{X}\left|\overline{\partial}u_{v}\right|_{\frac{1}{2}dd^{c}\phi}^{2}e^{-(\phi+v)}
\]
for \emph{some} solution $u_{v}$ of the corresponding $\overline{\partial}-$equation
and hence for $u_{v}$ as in formula \ref{eq:pf agmond:og} (we are
also using that $\overline{\partial}u$ and $\overline{\partial}v$
are supported in the bulk of $(X,\phi)$ to replace $\phi_{e}$ with
$\phi$ in the r.h.s). Using $\phi_{e}\leq\phi,$ \ref{eq:pf agmond: leibn}
and the ``parallelogram law'' then gives 
\[
\int_{X}\left|u\right|^{2}e^{-\phi}e^{v}\leq4\int_{X}(\left|\overline{\partial}u\right|_{dd^{c}\phi}^{2}+\left|\overline{\partial}vu\right|_{dd^{c}\phi}^{2})e^{-\phi_{e}}e^{v}
\]
 By assumption $(i)$ in the theorem the term in the r.h.s involving
$\overline{\partial}vu$ may be absorbed in the l.h.s. Finally, the
last statement in the theorem follows from the estimate in Remark
\ref{rem:perturbation by bounded}.\end{proof}
\begin{cor}
\label{cor:agmon asym}Let $L$ be a big line bundle and let $\phi$
be a $\mathcal{C}^{1,1}-$smooth weight on and $\omega_{n}$ a fixed
volume form on $X.$ Let $E$ be a given compact subset of the interior
of the bulk. Then there is a constant $C$ (depending on $E$ and
$F)$ such that the following holds. If $\psi_{k}$ is a sequence
of functions such that $d\psi_{k}$ is supported in the interior of
the bulk of $(X,\phi)$ and 
\[
(i)\,\,\left|\overline{\partial}\psi_{k}\right|_{dd^{c}\phi}^{2}\leq1/C\,\,\,\,\,(ii)\,dd^{c}\psi_{k}\leq\sqrt{k}dd^{c}\phi/C
\]
Then, for any sequence $f_{k}$ of smooth sections of $kL$ such that
$\overline{\partial}f_{k}$ is supported in the interior of the bulk
of $(X,\phi)$

\[
\left\Vert \Pi_{k}(f_{k})-f_{k}\right\Vert _{k\phi+\phi_{F}+\sqrt{k}\psi_{k}}^{2}\leq C\frac{1}{k}\left\Vert \overline{\partial}f_{k}\right\Vert _{k\phi+\phi_{F}+\sqrt{k}\psi_{k}}^{2},
\]
 where $\Pi_{k}$ is the Bergman projection with respect to $k\phi$
(formula \ref{def: K} and below). Moreover, the constant $C$ can
be taken to depend on $\phi_{F}$ only through an upper bound on the
$L^{\infty}-$norm $\left\Vert (\phi_{F}-\phi_{F_{0}})\right\Vert _{L^{\infty}(X)},$
where $\phi_{F_{0}}$ is a fixed smooth metric on $F.$ \end{cor}
\begin{proof}
Replacing $L$ with $kL+F-K_{X},$ $\phi$ with $k\phi+\phi_{F}$
and $v$ with $\sqrt{k}\psi_{k}$ the corollary follows from the previous
theorem using standard properties of orthogonal projections.\end{proof}
\begin{prop}
\label{pro:garding}The following local estimate holds for all $u$
which are $\mathcal{C}^{1}-$smooth (or more generally, Lipschitz
continuous): 
\begin{equation}
\sup_{\left|z\right|\leq Rk^{-1/2}}\left|u(z)\right|^{2}e^{-k\phi(z)}\leq C_{R}k^{n}(\int_{\left|z\right|\leq2Rk^{-1/2}}(\left|u\right|^{2}+\frac{1}{k}\left|\overline{\partial}u\right|^{2})e^{-k\phi}\omega_{n})\label{eq:garding}
\end{equation}
\end{prop}
\begin{proof}
This is a generalization of the uniformity statement in lemma \ref{lem:morse}.
It is proved in essentially the same way, by replacing the mean value
property of holomorphic functions used to prove lemma \ref{lem:morse}
by the general Cauchy formula for a smooth function $u.$ It is also
a consequence of Gårding's inequality - see (the proof of) lemma 3.1
in \cite{berm1} for a more general inequality.
\end{proof}

\section{\label{sec:Asymptotics-for-Bergman}Asymptotics for Bergman kernels
and correlations}

\subsection{Bergman kernels}

Recall that $\mathcal{H}(X,L)$ denotes the Hilbert space obtained
by equipping the vector space $H^{0}(X,L)$ with the inner product
corresponding to the norm induced by the weighted measure $(\phi,\omega_{n}).$
Let $(s_{i})$ be an orthonormal base for $\mathcal{H}(X,L).$ The
\emph{Bergman kernel} of the Hilbert space $\mathcal{H}(X,L)$ may
be defined as the holomorphic section 
\begin{equation}
K_{k}(x,y)=\sum_{i}s_{i}(x)\otimes\overline{s_{i}(y)}.\label{def: K}
\end{equation}
 of the pulled back line bundle $L\boxtimes\overline{L}$ over $X\times\overline{X}.$
To see that is independent of the choice of base $(s_{I})$ one notes
that $K_{k}$ represents the integral kernel of the orthogonal projection
$\Pi_{k}$ from the space of all smooth sections with values in $L$
onto $\mathcal{H}(X,L).$ 

The restriction of $K_{k}$ to the diagonal is a section of $L\otimes\overline{L}$.
Hence, its point wise norm $\left|K_{k}(x,x)\right|_{\phi}(=\left|K_{k}(x,x)\right|e^{-k\phi(x)})$
defines a well-defined function on $X$ that will be denoted by $\rho^{(1)}$
(and later identified with the\emph{ one point correlation function}):
\begin{equation}
\rho^{(1)}(x):=\sum_{i}\left|s_{i}(x)\right|_{k\phi}^{2}.\label{eq:def of B}
\end{equation}
It has the following well-known extremal property:
\begin{equation}
\rho^{(1)}(x):=\sup\left\{ \left|s(x)\right|_{\phi}^{2}:\,\,s\in\mathcal{H}(X,L),\,\left\Vert s\right\Vert _{\phi}^{2}\leq1\right\} \label{(I)extremal prop of B}
\end{equation}
 Moreover, integrating \ref{eq:def of B} shows that $\left|K_{k}(x,x)\right|_{\phi}$
is a ``dimensional density'' of the space $\mathcal{H}(X,L):$ 
\begin{equation}
\int_{X}\rho^{(1)}(x)\omega_{n}=\dim\mathcal{H}(X,L):=N\label{eq:dim formel for B}
\end{equation}
In section \ref{sub:Correlation-functions} we will consider a function
on the $N-$fold product $X^{N}$ that may, abusing notation slightly,
be written as 
\begin{equation}
\rho^{(N)}(x_{1},...,x_{N})=\det_{1\leq i,j\leq N}(K(x_{i},x_{j})e^{-\frac{1}{2}(\phi(x_{i})+\phi(x_{j}))}).\label{eq:bergman determinant}
\end{equation}
 To clarify the notation denote by $L^{\boxtimes N}$ the pulled-back
line bundle on $X^{N}$ with the weight induced by the weight $\phi$
on $L.$ Then the base $S=(s_{i})$ in $H^{0}(X,L)$ induces an element
$\det(S)$ in $H^{0}(X^{N},L^{\boxtimes N})$ whose value at $(x_{1},...,x_{N})$
is defined as the (Slater) determinant 
\begin{equation}
\det(S)(x_{1},.x_{N}):=\det_{1\leq i,j\leq N}(s_{i}(x_{i}))_{i,j}\in L_{x_{1}}\otimes\cdots\otimes L_{x_{N}}.\label{eq:slater det}
\end{equation}
In particular, its point-wise norm is a \emph{function} on $X^{N}$
which according to the following lemma may be locally written in the
form \ref{eq:bergman determinant}. The lemma also shows that after
division by $N!$ this function defines the density of a probability
measure on $X^{N}.$ Its proof is based on the following ``integrating
out'' property of the Bergman kernel $K$, which is a direct consequence
of the fact that $K$ is a projection kernel: 
\begin{equation}
\left|K(x,x)\right|_{\phi}=\int_{X}\left|K(x,y)\right|_{\phi}^{2}\omega_{n}(y)\label{eq:integr out}
\end{equation}
 
\begin{lem}
\label{lem:amazing id}The following identities hold point-wise: 
\[
\det_{1\leq i,j\leq N}(K(x_{i},x_{j})e^{-\frac{1}{2}(\phi(x_{i})+\phi(x_{j}))})=\left|\det(S)(x_{1},...,x_{N})\right|_{\phi}^{2}.
\]
 Integrating gives 
\[
\int_{X^{N}}\left|\det(S)(x_{1},...,x_{N})\right|_{\phi}^{2}\omega_{n}^{\otimes N}=N!.
\]
\end{lem}
\begin{proof}
The identities are formal consequences of the identity \ref{eq:integr out},
as is well-known in the random matrix literature. See for example
\cite{dei2}. The last identity can also be proved directly using
the following general identity \cite[Lemma 5.3]{b-b}: 
\begin{equation}
\int_{X^{N}}\left|\det(S)(x_{1},...,x_{N})\right|_{\phi}^{2}\omega_{n}^{\otimes N}=N!\det_{1\leq i,j\leq N}(\left\langle s_{i},s_{j}\right\rangle _{(\omega_{n},\varphi)})_{i,j},\label{eq:int det S in terms of Gram}
\end{equation}
given a base $(s_{i})$ in $H^{0}(X,L)$ and a bounded weight $\phi$
on $L.$
\end{proof}

\subsection{Scaling asymptotics of $K_{k}(x,y)$\label{sub:Scaling-asymptotics-of}
in the weak bulk}

In this section we fix a continuous metric $\omega$ on $X.$ Given
a point $x$ in $X$ we can take ``normal'' local coordinates $z$
centered at $x$ and a ``normal'' trivialization of $L,$ i.e such
that 
\begin{equation}
\,\,\,\omega_{x}=\frac{i}{2}\sum_{i=1}^{n}dz_{i}\wedge\overline{dz_{i}}+o(1)\,\,\,\phi(0)=d\phi(0)=0\label{eq: Intro metrics}
\end{equation}
Moreover, if the second partial derivatives of $\phi$ exist at $x$
then we may assume 
\[
(dd^{c}\phi)_{x}=\frac{i}{2\pi}\sum_{i=1}^{n}\lambda_{i}dz_{i}\wedge\overline{dz_{i}}
\]
Hence, the $\lambda_{i}$ are the eigenvalues of the curvature form
$dd^{c}\phi$ at $x$ w.r.t the metric $\omega$ and we denote the
corresponding diagonal matrix by $\lambda.$

For proofs of the following elementary local consequences of the regularity
properties of $\phi$ and $\phi_{e}$ see \cite{berm4}.
\begin{lem}
\label{lem:cons of reg}Given a point $x$ in $X$ and ``normal''
local coordinates $z$ centered at $x$ and a ``normal'' trivialization
of $L$ the following holds: 
\begin{equation}
\left|\phi(z)\right|\leq C\left|z\right|^{2},\label{eq:lemma cons of reg 1}
\end{equation}
where $C$ can be taken to be independent of the center $x$ on any
given compact subset of $X.$ Moreover, if the second partial derivatives
of $\phi$ exist at $z=0,$ then for any $\epsilon>0,$ there is a
$\delta>0$ such that 
\begin{equation}
(\left|z\right|\leq\delta\Rightarrow\left|\phi(z)-\sum_{i=1}^{n}\lambda_{i}\left|z_{i}\right|^{2}\right|\leq\epsilon\left|z\right|^{2}\label{eq:lemma cons of reg 2}
\end{equation}
and for any fixed positive number $R$ the following uniform convergence
holds when $k$ tends to infinity 
\begin{equation}
\sup_{\left|z\right|\leq R}\left|k\phi(\frac{z}{\sqrt{k}})-\sum_{i=1}^{n}\lambda_{i}\left|z_{i}\right|^{2}\right|\rightarrow0.\label{eq: lemma cons of reg 3}
\end{equation}
Finally, if the center $x$ is in the weak bulk, then for any $\epsilon>0,$
there is a $\delta>0$ such that 
\begin{equation}
(iii)\,\,\left|z\right|\leq\delta\Rightarrow\left|\phi_{e}(z)-\phi(z)\right|\leq\epsilon\left|z\right|^{2}\label{eq:lemma cons of reg 4}
\end{equation}

\end{lem}
The next lemma only uses\emph{ }local properties of holomorphic functions
and was called \emph{local holomorphic Morse inequalities} in \cite{berm1}.
See \cite{berm4} for the proof when the weight $\phi$ is merely
$C^{1,1}-$smooth.
\begin{lem}
\label{lem:morse}Fix a center $x$ in $X$ where the second derivatives
of the weight $\phi$ exist and normal coordinates $z$ centered at
$x.$ Then 
\[
\limsup_{k}k^{-n}\rho_{k}^{(1)}(z/k^{1/2})\leq\det_{\omega}(dd^{c}\phi)(x).
\]
Moreover, if $\left|z\right|\leq$$R$ then the l.h.s. above is uniformly
bounded by a constant $C_{R}$ which is independent of the center
$x.$
\end{lem}
Now we can prove the following \emph{lower} bound on the 1-point correlation
function in the weak bulk, which is a refinement of Lemma 4.4 in \cite{berm45}:
\begin{lem}
\label{lem:lower bd}Fix a center $x$ in the weak bulk and normal
coordinates $z$ centered at $x.$ Then 
\[
\liminf_{k}k^{-n}\rho_{k}^{(1)}(z/k^{1/2})\geq\det_{\omega}(dd^{c}\phi)(x)
\]
\end{lem}
\begin{proof}
\emph{Step1: construction of a smooth extremal $\sigma_{k}.$} Fix
a point $x$ in the weak bulk. First note that there is a \emph{smooth}
section $\sigma_{k}$ with values in $kL+F$ such that 
\begin{equation}
(i)\lim_{k\rightarrow\infty}\frac{\left|\sigma_{k}\right|_{k\phi}^{2}(z_{0}/\sqrt{k})}{k^{n}\left\Vert \sigma_{k}\right\Vert _{k\phi+\phi_{F}}^{2}}=(\frac{1}{2\pi})^{n}\det\lambda,\,\,\,(ii)\left\Vert \overline{\partial}\sigma_{k}\right\Vert _{k\phi_{e}+\phi_{F}}^{2}\leq Ce^{-k/C}\label{pf of lemma lower B: prop of sigma}
\end{equation}
To see this first take normal trivializations of $L$ and $F$ and
normal coordinates $z$ centered at $x$ (i.e. $x$ corresponds to
$z=0).$ Next, by scaling the coordinates $z$ we can assume that
\[
\omega{}_{x_{0}}=\frac{i}{2}\sum_{i=1}^{n}\frac{1}{\lambda_{i}}dz_{i}\wedge\overline{dz_{i}},\,\,\,(dd^{c}\phi)_{x_{0}}=\frac{i}{2\pi}\sum_{i=1}^{n}dz_{i}\wedge\overline{dz_{i}}
\]
Fix a smooth function $\chi$ which is equal to one when $\left|z\right|\leq\delta/2$
and supported where $\left|z\right|\leq\delta;$ the number $\delta$
will be assumed to be sufficiently small later on. Now $\sigma_{k}(z)$
is simply obtained as the local section with values in $L^{k}$ represented
by the function 
\[
\chi(z)e^{k(\bar{z}_{0}\cdot z-\frac{1}{2}\bar{z}_{0}\cdot z_{0})}
\]
 close to $z=0$ and extended by zero to all of $X.$ To see that
$(i)$ holds note first consider the numerator
\[
\left|\sigma_{k}\right|_{k\phi}^{2}(z_{0}/\sqrt{k})=e^{\bar{z}_{0}\cdot z_{0}}e^{-k\phi(z_{0}/\sqrt{k})}\rightarrow1,
\]
 when $k$ tends to infinity, using \ref{eq: lemma cons of reg 3}.
Next, write the the integrand in $k^{n}\left\Vert \sigma_{k}\right\Vert _{k\phi+\phi_{F},}^{2}$
in the form 
\[
\chi(z)^{2}k^{n}e^{-k(\left|z-z_{0}/\sqrt{k}\right|^{2}+(\phi(z)-\left|z\right|^{2}))}((\det\lambda)^{-1}+o(1))
\]
and decompose the region of integration according to the following
decomposition of the radial values: 
\begin{equation}
[0,\delta]=[0,R/\sqrt{k}]\bigsqcup[R/\sqrt{k},\delta],\label{eq:pf of lemma lower bd on b: regions}
\end{equation}
 where $R$ is a fixed large number. In the first region, we have
by \ref{eq: lemma cons of reg 3}, 
\[
\sup_{\left|z\right|\leq R/\sqrt{k}}\left|k(\phi(z)-\left|z\right|^{2})\right|\rightarrow0
\]
 Hence, performing the change of variables $z=z'/\sqrt{k}$ gives
\[
\lim_{k\rightarrow\infty}k^{n}\left\Vert \sigma_{k}\right\Vert _{k\phi+\phi_{F},[0,R/\sqrt{k}]}^{2}=(\det\lambda)^{-1}\int_{[0,R]}e^{-\left|z'-z_{0}\right|^{2}}(\frac{i}{2}\sum_{i=1}^{n}dz'_{i}\wedge\overline{dz'_{i}})^{n}/n!
\]
 As fort the second region in \ref{eq:pf of lemma lower bd on b: regions}
we have 
\begin{equation}
\left|z-z_{0}/\sqrt{k}\right|^{2}+(\phi(z)-\left|z\right|^{2})\geq\frac{1}{2}\left|z\right|^{2}\label{eq:eq:pf lemma lower bd on b: estim}
\end{equation}
for $R$ sufficiently large. Indeed, by \ref{eq:lemma cons of reg 2}
\[
\left|z\right|\leq\delta\Rightarrow\left|(\phi(z)-\left|z\right|^{2})\right|\leq\frac{1}{4}\left|z\right|^{2}.
\]
Moreover, 
\[
\left|z-z_{0}/\sqrt{k}\right|^{2}\geq\frac{1}{4}\left|z\right|^{2},
\]
 for all $k,$ if $R$ is sufficiently large. Hence, 
\[
k^{n}\left\Vert \sigma_{k}\right\Vert _{k\phi+\phi_{F},[R/\sqrt{k},\delta]}^{2}\leq\int_{[R/\sqrt{k},\delta]}k^{n}e^{-k\frac{1}{2}\left|z\right|^{2}}\rightarrow0,
\]
 since it is the tail of a convergent (Gaussian) integral (using the
change of variables $z=z'/\sqrt{k}$ again). Finally, letting first
$k$ and then $R$ tend to infinity finishes the proof of $(i)$ in
\ref{pf of lemma lower B: prop of sigma}.

Next, to prove $(ii)$ in \ref{pf of lemma lower B: prop of sigma},
first note that 
\begin{equation}
\left\Vert \overline{\partial}\sigma_{k}\right\Vert _{k\phi_{e}+\phi_{F}}^{2}\leq C'\int_{\delta/2\leq\left|z\right|\leq\delta}e^{-k(\left|z-z_{0}/\sqrt{k}\right|^{2}+(\phi(z)-\left|z\right|^{2})+\phi_{e}(z)-\phi(z)))}\omega_{n}(0)\label{eq:pf of lower bd b: exponent}
\end{equation}
 as follows from the definition of $\chi.$ Now take $\delta$ so
that, using \ref{eq:lemma cons of reg 2} and \ref{eq:lemma cons of reg 4}
,

\begin{equation}
\left|z\right|\leq\delta\Rightarrow\phi(z)+(\phi_{e}(z)-\phi(z))\geq\left|z\right|^{2}/4\label{pf of lemma lower bd B: delta}
\end{equation}
 for $\delta$ sufficiently small. Combining \ref{eq:eq:pf lemma lower bd on b: estim}
and \ref{pf of lemma lower bd B: delta} shows that the exponent in
\ref{eq:pf of lower bd b: exponent} is at most$-\frac{1}{4}k\left|z\right|^{2}$
which proves $(ii)$ in \ref{pf of lemma lower B: prop of sigma}.

\emph{Step2: perturbation of $\sigma_{k}$ to a holomorphic extremal
$\alpha_{k}.$} 

This step is just a repetition (word for word) of the corresponding
step in the proof of lemma 4.4 in \cite{berm45}. For completeness
we recall it briefly here. Equip $kL+F$ with a ``strictly positively
curved modification'' $\psi_{k}$ of the metric $k\phi_{e}+\phi_{F}$
as constructed in \cite{berm45}. Let $g_{k}=\overline{\partial}\sigma_{k}$
and let $\alpha_{k}$ be the following holomorphic section 
\[
\alpha_{k}:=\sigma_{k}-u_{k},
\]
 where $u_{k}$ is the solution of the $\overline{\partial}$-equation
in the Hörmander-Kodaira theorem \ref{thm:horm-kod-for phi} with
$g_{k}=\overline{\partial}\sigma_{k}.$ Using properties of $\phi_{e}$
on then obtains the estimate
\begin{equation}
\left\Vert u_{k}\right\Vert _{k\phi+\phi_{F}}\leq C\left\Vert g_{k}\right\Vert _{k\phi_{e}+\phi_{F}}\label{eq:uk}
\end{equation}
and then $(ii)$ in \ref{pf of lemma lower B: prop of sigma} in the
right hand side gives 
\[
(a)\,\left\Vert u_{k}\right\Vert _{k\phi+\phi_{F}}\leq Ce^{-k/C},\,\,\,(b)\,\left|u_{k}\right|_{k\phi+\phi_{F}}^{2}(x)\leq C'k^{n}e^{-k/C'},
\]
where $(b)$ is a consequence of $(a)$ (using Prop \ref{pro:garding}
at $z=0).$ Combining $(a)$ and $(b)$ with $(i)$ in \ref{pf of lemma lower B: prop of sigma}
then proves that $(i)$ in \ref{pf of lemma lower B: prop of sigma}
holds with $\sigma_{k}$ replaced by the holomorphic section $\alpha_{k}.$
By the definition of $\rho_{k}^{(1)}$ this finishes the proof of
the lemma.
\end{proof}
Before turning to the proof of Theorem \ref{thm:intro bulk univ}
we also recall the following uniform estimate (which follows from
lemma \ref{lem:morse} precisely as in lemma 5.2 $(i)$ in \cite{berm3}):
\begin{lem}
\label{lem:uniform bd on k}Fix a center $x$ in $X$ and normal coordinates
$z$ and $w$ centered at $x$ with $z,w$ contained in a fixed compact
set. Then 
\[
k^{-2n}\left|K_{k}(z/k^{1/2},w/k^{1/2}))\right|_{k\phi+\phi_{F}}^{2}\leq C
\]
for some constant independent of the center $x$ in $X.$
\end{lem}

\subsubsection{Proof of Theorem \ref{thm:intro bulk univ}.}

Fix a point $x_{0}$ in $X$ and take coordinates $z$ and $w$ centered
at $x$ and normal trivializations of $L$ and $F$ as in the proof
of the previous lemma, inducing corresponding trivializations around
$(x,x)$ in $X\times X.$ Consider the holomorphic functions $f_{k}(z,w)=k^{-n}K_{k}(k^{-1/2}z,k^{-1/2}\bar{w})$
and $f(z,w)=\det_{\omega}(dd^{c}\phi)(x_{0})e^{zw}$ on the polydisc
on $\Delta_{R}$ of radius $R$ centered at the origin in $\C^{2n}.$
By lemma \ref{lem:uniform bd on k}: 
\begin{equation}
\sup_{\Delta_{R}}\left|f_{k}\right|\leq C_{R},\label{eq:pf of thm bulk univ}
\end{equation}
 Moreover, combining the upper and lower bounds in lemma \ref{lem:morse}
and lemma \ref{lem:lower bd}, respectively, shows that $f_{k}$ tends
to $f$ on $M:=\{(z,\bar{z})\in\Delta_{R}\}.$ Now, by the bound \ref{eq:pf of thm bulk univ}
$f_{k}$ has a convergent subsequence converging uniformly on $\Delta_{R}$
to a holomorphic function $f_{\infty}$ where necessarily $f_{\infty}=f$
on $M.$ But since $M$ is a maximally totally real submanifold it
follows that $f_{\infty}=f$ everywhere on $\Delta_{R}.$ Since, the
argument can be repeated for any subsequence of $f_{k}$ this proves
the uniform convergence in the theorem. Finally, the convergence of
higher derivatives is a standard consequence of Cauchy estimates.
\begin{rem}
In fact, Theorem \ref{thm:intro bulk univ} also follows in a more
or less formal way (using the method in \cite{berm3}) from combining
Lemma \ref{lem:morse} with the the special case of Lemma \ref{lem:lower bd}
obtained by setting $z=0$ (which was obtained in \cite{berm45}).
But the present method is more explicit and hence gives a better control
on the convergence, which might be useful in other contexts.
\end{rem}

\subsection{Off-diagonal decay of $K_{k}(x,y)$}

The next theorem is a refined version of Theorem \ref{thm:intro decay}
stated in the introduction (the dependence on the line bundle $F$
will be important in the proof of Theorem \ref{thm:conv of laplace}).
\begin{thm}
\label{thm:off-diagon decay of k}Let $L$ be a big line bundle and
$K_{k}$ the Bergman kernel of the Hilbert space $\mathcal{H}(kL+F).$
Let $E$ be a compact subset of the interior of the bulk. Then there
is a constant $C$ (depending on $E)$ such that the following estimate
holds for all pairs $(x,y)$ such that either $x$ or $y$ is in $E:$
\[
k^{-2n}\left|K_{k}(x,y)\right|_{k\phi+\phi_{F,t}}^{2}\leq Ce^{-\sqrt{k}d(x,y)/C}
\]
for all $k,$ where $d(x,y)$ is the distance function with respect
to a fixed smooth metric $\omega$ on $X.$ Moreover, fixing a smooth
reference weight $\phi_{F_{0}}$on $L$ the constant $C$ can be taken
to only depend on the continuous weight $\phi_{F}$ via an upper bound
on the $L^{\infty}-$norm $\left\Vert (\phi_{F}-\phi_{F_{0}})\right\Vert _{L^{\infty}(X)}.$\end{thm}
\begin{proof}
Fix a point $x$ in $X$ and take an element $s_{k}$ in $\mathcal{H}_{k}$
such that
\begin{equation}
\left|s_{k}\right|^{2}e^{-k\phi}=\left|K_{k}(x,\cdot)\right|^{2}e^{-k\phi(x)}e^{-k\phi(\cdot)}\label{eq:pf of thm decay: K as s}
\end{equation}
Next, fix a point $y$ in the set $E$ appearing in the formulation
of the theorem and ``normal'' local coordinates $z$ centered at
$y$ and a ``normal'' trivialization of $L$ (see the beginning
of the section). In particular, $\phi(0)=\overline{\partial}\phi(0)=0.$
Identify $s_{k}$ with a local holomorphic function in the $z-$variable.
By the mean value property of holomorphic functions 
\[
s_{k}(0)=\int\chi_{k}s_{k},
\]
 where $\chi_{k}=c_{n}k^{n}\chi(\sqrt{k}z)$ has unit mass and is
expressed in terms of a radial smooth function $\chi$ supported on
the unit-ball (so that $\chi_{k}$ is supported on the scaled unit
ball of radius $1/\sqrt{k}).$ Writing $\chi_{k\phi}:=\chi_{k}e^{k\phi(x)}$
the relation \ref{eq:pf of thm decay: K as s} gives, 
\[
\left|s_{k}\right|_{k\phi}(y)=\left|\left\langle \chi_{k\phi},s_{k}\right\rangle _{k\phi}\right|=\left|\Pi_{k}(\chi_{k\phi})(x)\right|_{k\phi}(x)
\]
 using the definition of $s_{k}$ in the last equality. Decomposing
$\Pi_{k}(\chi_{k\phi})=\chi_{k\phi}+(\Pi_{k}(\chi_{k\phi})-\chi_{k\phi})$
and applying Theorem \ref{thm:agmond} combined with proposition \ref{pro:garding}
now yields the following estimate 
\begin{equation}
\left|s_{k}\right|_{k\phi+\sqrt{k}\psi_{k}}(y)\leq\left|\chi_{k\phi}\right|_{k\phi+\sqrt{k}\psi_{k}}(x)+Ck^{(n-1)/2}\left\Vert \overline{\partial}\chi_{k\phi}\right\Vert _{k\phi+\sqrt{k}\psi_{k}}\label{eq:pf decay K: est 1}
\end{equation}
 for any function $\psi_{k}$ satisfying the assumptions in Theorem
\ref{thm:agmond}. The idea now is take $\psi_{k}$ to be comparable
to the distance to $x.$ In the following we will denote by $R$ a
sufficiently large (but fixed constant).

\emph{Case 1: $d(x,y)\geq1/R.$} Set $\psi_{k}=\psi$ for a fixed
smooth function $\psi$ on $X$ such that $\psi(\cdot)=1/R$ when
$d(x,\cdot)\geq1/(2R)$ and $\psi(\cdot)=0$ for when $d(x,\cdot)\leq1/(4R).$
For $R>>1$ (but fixed) the assumptions on $\psi_{k}$ in Theorem
\ref{thm:agmond} are clearly satisfied (using that $y$ is in the
interior of the bulk). Hence, the estimate \ref{eq:pf decay K: est 1}
gives 
\[
\left|s_{k}\right|_{k\phi}^{2}e^{\sqrt{k}/C}(y)\leq0+Ck^{n}\frac{1}{k}\left\Vert \overline{\partial}\chi_{k\phi}\right\Vert _{k\phi+0}^{2}\leq C'k^{2n}
\]
using that $\psi=0$ on the support of $\chi_{k\phi}$ and that $\left|k\overline{\partial}\phi\right|^{2}$
is uniformly bounded there (since $\overline{\partial}\phi$ is assumed
to be Lipschitz continuous and vanishing when $z=0).$ Since by definition
$s_{k}$ is related to $K_{k}$ by the relation \ref{eq:pf of thm decay: K as s}
this proves the theorem in this case.

\emph{Case 2: $d(x,y)\leq1/R.$} In this case we may assume that $x$
is contained in the fixed coordinate neighborhood of $y.$ By a translation
of the coordinates $z$ we now assume that they are centered at $x.$
Set 
\[
\psi_{k}(z)=\frac{1}{R}\kappa(\left|z\right|^{2}+1/k)^{1/2}
\]
 where $\kappa$ corresponds to a smooth function on $X$ which is
equal to one on the ``ball'' $\{d(,y)\leq2/C\}$ and is supported
in the set $E.$ Accepting for for the moment that the assumptions
on $\psi_{k}$ in Theorem \ref{thm:agmond} are satisfied, the inequality
\ref{eq:pf decay K: est 1} gives (with $z\leftrightarrow y)$
\[
\left|s_{k}\right|_{k\phi}^{2}e^{\sqrt{k}(\left|z\right|^{2}+1/k)^{1/2})}(z)\leq\left|\chi_{k\phi}\right|_{k\phi+1}^{2}(x)+C'k^{n}\frac{1}{k}\left\Vert \overline{\partial}\chi_{k\phi}\right\Vert _{k\phi+1}^{2}\leq C''k^{2n}
\]
using that $\sqrt{k}\psi_{k}\geq\sqrt{k}/\sqrt{k}$ on the support
of $\chi_{k\phi}$ in the first inequality. In particular, 
\[
\left|s_{k}\right|_{k\phi}^{2}(z)\leq C'k^{2n}e^{-\sqrt{k}\left|z\right|}
\]
which proves the theorem, since the distance function $d(\cdot,y)$
is comparable, close to $y,$ with the distance function induced by
the local Euclidean metric. 

Next, let us check that the assumptions on $\psi_{k}$ in Theorem
\ref{thm:agmond} are indeed satisfied. Differentiating gives 
\begin{equation}
\overline{\partial}\psi_{k}=\frac{1}{R}(\overline{\partial}\kappa\cdot(\left|z\right|^{2}+1/k)^{1/2}-\kappa\frac{zd\bar{z}}{2(\left|z\right|^{2}+1/k)^{1/2}})\label{eq:pf them decay: leibn}
\end{equation}
Hence, 
\begin{equation}
\left|\overline{\partial}\psi_{k}\right|\leq\frac{1}{R}(C'+C''\sqrt{k})\label{eq:pf thm decay: bf on deriv}
\end{equation}
 so that $(i)$ in Theorem \ref{thm:agmond} holds for $R>>1.$ Next,
note that $f_{k}:=(\left|z\right|^{2}+1/k)^{1/2}$ is a psh function.
Hence, formula \ref{eq:pf them decay: leibn} combined with Leibniz
rule gives 
\[
\partial\overline{\partial}\psi_{k}\geq\partial\overline{\partial}\kappa\cdot f_{k}+\partial\kappa\wedge\overline{\partial}f_{k}+\overline{\partial}\kappa\wedge\partial f_{k}
\]
and \ref{eq:pf thm decay: bf on deriv} (which clearly also holds
when $\psi_{k}$ is replaced by $f_{k})$ then shows that assumption
$(ii)$ in Theorem \ref{thm:agmond} holds, as well (even without
taking $R$ large).

Finally, the last statement in the theorem, concerning the dependence
on $\phi_{F},$ follows immediately from writing $\phi_{F}=f+\phi_{F_{0}}$
and repeating the previous proof with $k\phi$ replaced with $k\phi+f$
and using the $L^{2}-$estimate in Remark \ref{rem:perturbation by bounded}. 
\end{proof}

\subsection{Fluctuations}
\begin{thm}
\label{thm:fluct of bergman}Let $L$ be a big line bundle and $K_{k}$
the Bergman kernel of $\mathcal{H}(X,kL+F).$ Let $u$ be a Lipschitz
continuous function on $X.$ Then 
\[
\liminf_{k\rightarrow\infty}\frac{1}{2}\int_{X\times X}k^{-(n-1)}\left|K_{k}(x,y)\right|_{k\phi+\phi_{F}}^{2}(u(x)-u(y))^{2}\geq\left\Vert du\right\Vert _{(S,\omega_{\phi})}^{2}
\]
where equality holds, with $\liminf$ replace by $\lim,$ if $u$
is supported in a compact subset of the bulk. Moreover, if $\phi_{F}$
satisfies the assumptions in the previous theorem, then the left hand
side above is uniformly bounded by a constant only depending on $\phi_{F}$
through the $L^{\infty}-$norm of $\phi_{F}-\phi_{F_{0}}.$ \end{thm}
\begin{proof}
Let us start by the first point in the theorem, i.e. the case when
$u$ is compactly supported in the bulk.. Denote by $E$ the support
of $u.$ First note that the integrand vanishes if both $x$ and $y$
are in $X-E.$ We rewrite the integral above as follows: 
\[
2I_{k}:=\int_{E\times X\cup X\times E}\left|k^{1/2}(u(y)-u(x))\right|^{2}k^{-n}\left|K_{k}(x,y)\right|^{2}e^{-k\phi(x)}e^{-k\phi(x)}\omega_{n}(x)\wedge\omega_{n}(y),
\]
 Decompose the integral above as $A_{k,R}+B_{k,R}+C_{k,R}$ according
to the following three regions:

\emph{First region $(1\leq d(x,y))$:} By symmetry we may assume that
$x\in E.$ But then Theorem \ref{thm:off-diagon decay of k} shows
that $A_{k}$ tends to zero only using that $u$ is bounded. 

\emph{Second region $(Rk^{-1/2}\leq d(x,y)\leq1):$} Again, by symmetry
we may assume that $x\in E.$ Since $u$ is Lipschitz continuous $\left|u(y)-u(x)\right|\leq Cd(x,y).$
Hence, by Theorem \ref{thm:off-diagon decay of k} 
\[
B_{k,R}\leq C\int_{Rk^{-1/2}\leq\{d(x,y)\leq1}\left|\sqrt{k}d(x,y)\right|^{2}k^{n}e^{-\sqrt{k}d(x,y)/C}\omega_{n}(x)\wedge\omega_{n}(y).
\]
Performing a change of variables (with $y$ fixed) then gives 
\begin{equation}
I_{k}\leq C\int_{X}(\int_{2\sqrt{k}\geq\left|\zeta\right|\geq R/2}\left|\zeta\right|^{2}e^{-\left|\zeta\right|}d\zeta...)\omega_{n}(x)\rightarrow0,\label{eq:pf of thm agmond sec ref}
\end{equation}
 when first $k$ and then $R$ tends to infinity. 
\end{proof}
\emph{Third region $(d(x,y)\leq Rk^{-1/2}):$} 

By the previous discussion only the third region gives a contribution
to the asymptotics of the integrals $I_{k}:$ 
\[
\lim_{k\rightarrow\infty}I_{k}:=0+0+\lim_{R\rightarrow\infty}\lim_{k\rightarrow\infty}C_{k,R},
\]
assuming that the last limits exist, as well be shown next. To this
end fix $R>0$ and note that, using a partition of unity we may as
well replace the total region of integration $X\times X$ by $U\times U,$
where $U$ is a given local coordinate neighborhood. Moreover, the
third region $C_{k,R}$ may as well be replaced by the region $C'_{k,R}$
defined by $|x-y|\leq Rk^{-1/2},$ expressed in terms of the Euclidean
distance on $U$ (just using that $A^{-1}|x-y|\leq d(x,y)\leq A|x-y|$
on $U$ for some positive constant $A).$ Upon removing a set of measure
zero we may also assume that $x$ is in the bulk (since $E$ is a
compact set in the interior of the bulk) and that the first order
derivatives of $u$ exist at $x.$ Now take ``normal coordinates''
$z$ and trivializations of $L$ and $F$ centered at $x.$ Then the
integral over $\{x\}\times Y$ in $C'_{k,R}$ may be written as 
\begin{equation}
\int_{\left|z\right|\leq R}g_{k}(x,z)\omega_{n}(k^{-1/2}z),\label{eq:of fluctiations of k: integrand}
\end{equation}
 where, $g_{k}(x,z):=$ 
\[
=\left|k^{1/2}(u(k^{-1/2}z)-u(0))\right|^{2}k^{-2n}\left|K_{k}(0,k^{-1/2}z)\right|^{2}e^{-k\phi(k^{-1/2}z)}e^{-k\phi(0)}\omega_{n}(k^{-1/2}z)
\]
(using the change of variables $z\rightarrow k^{-1/2}z).$ Since $u$
is assumed to be Lipschitz continuous and differentiable at $z=0$
we have 
\[
\sup_{\left|z\right|\leq R}\left|\left(k^{1/2}(u(k^{-1/2}z)-u(0))\right)-(\sum_{i=1}^{n}a_{i}z_{i}+\overline{a_{i}}\overline{z_{i}})\right|\rightarrow0,\,\,a_{i}:=\frac{\partial u}{\partial z_{i}}(0)
\]
By the scaling asymptotics in Theorem \ref{thm:intro bulk univ} and
Lemma \ref{lem:uniform bd on k} and the Lipschitz assumption on $u$
we have 
\[
|g_{k}(x,z)|\leq A_{R}
\]
and
\[
\lim_{k\rightarrow\infty}g_{k}(x,z)=\int_{\left|z\right|\leq R}\left|\sum_{i=1}^{n}a_{i}z_{i}+\overline{a_{i}}\overline{z_{i}})\right|^{2}(\frac{\det\lambda}{\pi^{n}})^{2}e^{-\left\langle \lambda z,z\right\rangle }\frac{i}{2}dz_{1}\wedge d\bar{z}_{1}\wedge\cdots
\]
As a consequence, computing the Gaussian integrals gives 
\[
\lim_{R\rightarrow\infty}\lim_{k\rightarrow\infty}g_{k}(x,z)=(\frac{\det\lambda}{\pi^{n}})\sum_{i}2\left|\frac{\partial}{\partial z_{i}}u(0)\right|^{2}\lambda_{i}^{-2}c_{n},
\]
 where 
\[
c_{n}=(\int_{0}^{\infty}se^{-s}ds)^{n}=-\frac{d}{dt}\mid_{t=1}\int_{0}^{\infty}e^{-ts}ds=1
\]
 Hence, by the dominated convergence 
\[
I=\frac{1}{2}\lim_{R\rightarrow\infty}\lim_{k\rightarrow\infty}C_{k,R}=\frac{1}{\pi}\int_{X}\left|\partial u\right|_{(dd^{c}\phi)}^{2}(dd^{c}\phi)^{n}/n!,
\]
 which concludes the proof of the convergence in theorem. To prove
the last statement of the theorem just note that the integrand may,
as above, be estimated from above by 
\[
C\left|\sqrt{k}d(x,y)\right|^{2}k^{n}e^{-\sqrt{k}d(x,y)/C},
\]
 where $C$ only depends on the $L^{\infty}-$norm of $|\phi_{F}-\phi_{F_{0}}|,$
according to Theorem \ref{thm:fluct of bergman}.  Integrating over
$x$ and $y$ then concludes the proof, as above.

Finally, for a general Lipschitz continuous $u$ the lower bound on
the second point of the theorem follows by restricting the integration
to the third region above with $x$ restricted to the weak bulk. Indeed,
letting first $k\rightarrow\infty$ using the scaling asymptotics
in Theorem \ref{thm:intro bulk univ} as above together with Fatou's
lemma and then letting $R\rightarrow\infty,$ using the monotone convergence
theorem, gives the desired lower bound.

\section{\label{sec:Asymptotics-for-linear}Asymptotics for linear statistics }

Let us first recall the setup in section \ref{sec:Asymptotics-for-Bergman}.
A line bundle $L\rightarrow X$ and a pair $(\phi,\omega_{n})$ induces
a Hilbert space $\mathcal{H}(X,L)$ of dimension $N$ with associated
Bergman kernel $K(x,y).$ Recall also that, in general, a subindex
$k$ on an object indicates that it is defined with respect to $(kL,k\phi).$
Hence, we will set $k=1$ in the following definitions.

We define the associated ensemble $(X^{N},\gamma)$ by letting $\gamma$
be the probability measure with the following density: 
\[
\mathcal{P}(x_{1},...,x_{N}):=\frac{1}{N!}\det(K(x_{i},x_{j})e^{-\frac{1}{2}(\phi(x_{i})+\phi(x_{j}))}).
\]
 By lemma \ref{lem:amazing id} this is indeed a well-defined probability
measure. Note that the ensemble is symmetric in the sense that $\mathcal{P}(x_{1},...,x_{N})$
is invariant under permutations of the components $x_{i}.$

\subsection{\label{sub:Correlation-functions}Correlation functions}

Next, we recall the general formalism of correlation functions. But
it should be pointed out that in the present paper we will mainly
consider the correlation functions in formula \ref{eq:1 pt and 2 pt correl}
below, that the reader could also take as definitions. 

For a general symmetric ensemble $(X^{N},\gamma)$ the \emph{$m-$point
correlation measures} on $X^{m}$ may be defined as $N!/(N-m)!$ times
the pushforward of $\gamma$ to $X^{m}$ under the projection $(x_{1},...x_{N})\mapsto(x_{1},...,x_{m})$
(i.e. the $m-$dimensional \emph{marginal} of $\gamma$). The \emph{$m-$point
correlation} \emph{functions} $\rho^{(m)}$ on $X^{m}$ are then defined
as the corresponding densities. As is well-known \cite{dei2,s1} the
fact that the defining kernel $K$ of the process represents an orthogonal
projection operator leads to the following quite remarkable identities
in the present context:
\[
\rho^{(m)}(x_{1},...,x_{m})=\det_{1\leq i,j\leq m}(K(x_{i},x_{j})e^{-\frac{1}{2}(\phi(x_{i})+\phi(x_{j}))})
\]
 A crucial role in the present paper is played by the so called \emph{connected
$2-$point correlation} \emph{function} $\rho^{(2),c}$ which may
be defined by
\[
\rho^{(2).c}(x,y):=\rho^{(2)}(x,y)-\rho^{(1)}(x)\rho^{(1)}(y)
\]
 Hence, $\rho^{(1)}$ and $\rho^{(2).c}$ may be simply expressed
as 
\begin{equation}
\rho^{(1)}(x)=\left|K(x,x)\right|_{\phi},\,\,\,\,\,\rho^{(2).c}(x,y)=-\left|K(x,y)\right|_{\phi}^{2}.\label{eq:1 pt and 2 pt correl}
\end{equation}
 
\begin{rem}
The present setup is essentially a special case of the general formalism
of determinantal random point processes \cite{s1,h-k-p,j2}. It falls
into the class of such processes where the correlation kernel is the
integral kernel of an orthogonal projection operator. 
\end{rem}

\subsection{Linear statistics}

A given (measurable) function $u$ on $(X,\omega_{n})$ induces the
following random variable $\mathcal{N}[u]$ on $(X^{N},d\mathcal{P}):$
\[
\mathcal{N}[u](x_{1},...,x_{N}):=u(x_{1})+....+u(x_{N}).
\]
 Hence, if $u$ is the characteristic function of a set $\Omega$
in $X,$ then $\mathcal{N}[u](x_{1},...,x_{n})$ simply counts the
number of $x_{i}$ contained in $\Omega.$ However, we will mainly
focus on the case when $u$ is continuous. For a given random variable
$\mathcal{X}$ we will write its \emph{fluctuation} as the random
variable
\[
\widetilde{\mathcal{X}}:=\mathcal{X}-\E(\mathcal{X}),
\]
 so that $\E(\widetilde{\mathcal{X}})=0.$ Recall that the variance
of a random variable $\mathcal{X}$ is defined as
\[
\textrm{Var}(\mathcal{X}):=\E((\widetilde{\mathcal{X}})^{2})
\]
The following lemma is also essentially well-known.
\begin{lem}
\label{lem:formula for exp and var}The following formulas for the
expectation and variance of $\mathcal{N}_{k}[u]$ hold:
\[
(i)\,\,\,\E_{\phi+tu}(\mathcal{N}[u])=-\frac{d}{dt}\log\E_{\phi+tu}(e^{-t\mathcal{N}[u]})=\int_{X}\left|K_{\phi+tu}(x,x)\right|_{\phi+tu}u(x)
\]
and 
\[
(ii)\,\,\,\textrm{Var}_{\phi+tu}(\mathcal{N}[u]))=\frac{d^{2}}{d^{2}t}\log\E_{\phi+tu}(e^{-t\mathcal{N}[u]})=
\]
\[
=\frac{1}{2}\int_{X\times X}\left|K_{\phi+tu}(x,y)\right|_{\phi+tu}^{2}(u(x)-u(y))^{2}\omega_{n}(x)\wedge\omega_{n}(y)
\]
\end{lem}
\begin{proof}
Without loss of generality we may as well calculate the derivatives
at $t=0$ (indeed, at a general $t=t_{0}$ one then rewrites $\phi+(t_{0}+\epsilon)u=(\phi+t_{0}u)+\epsilon u$
and applies the previous case with $\phi$ replaced by $\phi+t_{0}u).$
Set $f(t):=-\log\E(e^{-t\mathcal{N}[u]})$ Then it follows immediately
that 
\[
\frac{d}{dt}_{|t=0}f(t)=\int_{X^{N}}\sum_{i=1}^{N}u(x_{i})\rho^{(N)}(x_{1},...,x_{N})\omega_{n}^{\otimes N}=\int_{X}u\rho^{(1)}\omega_{n},
\]
 which, combined with formula \ref{eq:1 pt and 2 pt correl} proves
the item $(i).$ Similarly, 
\[
\frac{d^{2}f(t)}{d^{2}t}_{|t=0}=\int_{X^{N}}\sum_{1\leq i,j\leq N}u(x_{i})u(x_{j})\rho^{(N)}(x_{1},...,x_{N})\omega_{n}^{\otimes N}
\]
and hence splitting the sum over the indices $(i,j)$ where $i=j$
and $i<j$ gives 
\[
\frac{d^{2}f(t)}{d^{2}t}_{|t=0}=\int_{X}u^{2}\rho^{(1)}\omega_{n}+\int_{X^{2}}u(x)u(y)\rho^{(2)}(x,y)\omega_{n}
\]
Invoking formula \ref{eq:1 pt and 2 pt correl} for $\rho^{(2)}(x,y)$
thus gives that 
\[
\frac{d^{2}f(t)}{d^{2}t}_{|t=0}=\int_{X}u^{2}|K(x,x)|_{\phi}\omega_{n}+\int_{X^{2}}u(x)u(y)\left(|K(x,x)|_{\phi}|K(y,y)|_{\phi}-|K(x,y)|_{\phi}^{2}\right)
\]
Under the normalization that $\E_{\phi}(\mathcal{N}[u]):=\int u\rho^{1}\omega_{n}=0$
this means that 
\[
\frac{d^{2}f(t)}{d^{2}t}_{|t=0}=\int_{X}u^{2}|K(x,x)|_{\phi}\omega_{n}-\int_{X^{2}}u(x)u(y)|K(x,y)|_{\phi}^{2}.
\]
The proof is now concluded by first rewriting $u(x)u(y)=-(u(x)-u(y))^{2}/2+u(x)^{2}/2+u(y)^{2}/2$
and then integrating over first $x$ and then $y$ and using that
(by the reproducing property) $|K(x,x)|_{\phi}=\int_{X}|K(x,y)|_{\phi}^{2}\omega_{n}(y).$ \end{proof}
\begin{rem}
\label{rem:gram}Let $(s_{i})$ be an orthonormal base for $H^{0}(X,L)$
w.r.t. $(\phi,\omega_{n}).$ Then $\E(e^{-t\mathcal{N}[u]})$ may
be alternatively expressed as a Gram determinant: 
\begin{equation}
\E(e^{-t\mathcal{N}[u]})=\det\left(\left\langle s_{i},s_{j}\right\rangle _{\phi+tu}\right)_{i,j}\label{eq:expect in terms of gram}
\end{equation}
 and hence form the point of view of Kähler geometry the functional
$u\mapsto-\log\E(e^{-t\mathcal{N}[u]})$ can be viewed as a Donaldson
$\mathcal{L}_{k}-$functional (see \cite{don1,b-b,berman4 och halv}
and references therein). Formula \ref{eq:expect in terms of gram}
follows immediately from writing
\[
\E(e^{-t\mathcal{N}[u]})=\frac{\int_{X^{N}}\left|\det(S)(x_{1},...,x_{N})\right|_{\phi+tu}^{2}\omega_{n}^{\otimes N}}{\int_{X^{N}}\left|\det(S)(x_{1},...,x_{N})\right|_{\phi}^{2}\omega_{n}^{\otimes N}}.
\]
 and applying the identity \ref{eq:int det S in terms of Gram} to
the weights $\phi$ and $\phi+tu.$ \end{rem}
\begin{prop}
\label{pro:varians est}Suppose that $u$ is a bounded function on
$X$ and $(\phi,\mu)$ is a general weighted measure. Then 
\[
(i)\,\,\textrm{Var}_{k}(\mathcal{N}[u]))=O(k^{n})
\]
Moreover, if $(\phi,\omega_{n})$ is strongly regular and $u$ continuous,
then 
\[
(ii)\,\,\textrm{Var}_{k}(\mathcal{N}[u]))=o(k^{n}).
\]
\end{prop}
\begin{proof}
By $(ii)$ in lemma \ref{lem:formula for exp and var} 
\[
\textrm{Var}_{k}(\mathcal{N}[u]))=\frac{1}{2}\int_{X\times X}\left|K_{k}(x,y)\right|_{k\phi}^{2}(u(x)-u(y))^{2}\omega_{n}(x)\wedge\omega_{n}(y)
\]
 The first item of the proposition follows immediately, since $u$
is assumed bounded, from combining \ref{eq:integr out} and \ref{eq:dim formel for B}
and using that $N_{k}=O(k^{n})$ for \emph{any} line bundle $L$.
The second item follows from \cite{berm45} where it is shown that
\[
\int k^{-n}\left|K_{k}(x,y)\right|_{k\phi}^{2}f(x)g(y)\omega_{n}(x)\wedge\omega_{n}(y)\rightarrow\int_{X}fg\mu_{\phi_{e}},
\]
for any continuous functions $f,g.$
\end{proof}

\subsection{A law of large numbers (proof of Thm \ref{thm:conv in prob})}

By $(i)$ in Lemma \ref{lem:formula for exp and var} and \cite[Thm B]{b-b-w}:
\[
\E_{k}(k^{-n}\mathcal{N}[u])=\int_{X}\left|K_{k}\right|_{k\phi}u\omega_{n}\rightarrow\int_{X}u\mu_{\phi_{e}}.
\]
Moreover, by $(i)$ in the previous proposition 
\[
\textrm{Var}_{k}(k^{-n}\mathcal{N}[u]))=O(k^{-n})\rightarrow0.
\]
 Hence, the theorem follows directly from Chebishevs inequality, just
like in the usual proof of the classical weak law of large numbers.

\subsection{A central limit theorem (proof of Thm \ref{thm:conv of laplace}).}
\begin{proof}
We start by taking $t\in\R.$ Let $\mathcal{F}_{k}(t):=-\log\E_{k}(e^{-tk^{-(n-1)/2}\widetilde{\mathcal{N}_{k}}[u]}).$
By $(i)$ in Lemma \ref{lem:formula for exp and var}
\begin{equation}
\frac{d\mathcal{F}_{k}(t)}{dt}_{t=0}=k^{-(n-1)/2}\E_{k}(\widetilde{\mathcal{N}_{k}})=0,\label{eq:vanishing of der}
\end{equation}
 using the definition of $\widetilde{\mathcal{N}_{k}}$ in the last
equality. Moreover, by $(ii)$ in Lemma \ref{lem:formula for exp and var}
\[
\frac{d^{2}\mathcal{F}_{k}(t)}{d^{2}t}=-k^{-(n-1)}\frac{1}{2}\int_{X\times X}\left|K_{k\phi+th_{k}}(x,y)\right|_{k\phi+th_{k}}^{2}(h_{k}(x)-h_{k}(y))^{2}
\]
where $h_{k}=u-c_{k}$ with $c_{k}=\E_{k}(\mathcal{N}_{k}).$ Next,
note that the map $\psi\mapsto\left|K_{\psi}(x,y)\right|_{\psi}^{2}$
is clearly invariant under $\psi\rightarrow\psi+c$ for any constant
$c.$ Hence, we get 
\[
\frac{d^{2}\mathcal{F}_{k}(t)}{d^{2}t}=-\frac{1}{2}\int_{X\times X}\left|K_{k\phi+tu}(x,y)\right|_{k\phi+tu}^{2}(u(x)-u(y))^{2}
\]
 Applying Theorem \ref{thm:fluct of bergman} to $kL+F$ where $F$
is the trivial holomorphic line bundle equipped with the weight $k^{-(n-1)/2}tu$
(taking for example $\phi_{F_{0}}\equiv0)$ gives 
\begin{equation}
\lim_{k\rightarrow\infty}\frac{d^{2}\mathcal{F}_{k}(t)}{d^{2}t}=-\left\Vert du\right\Vert _{dd^{c}\phi}^{2}\label{eq:ptwise conv of second deriv}
\end{equation}
 for all $t.$ Using that the second order derivatives of $\mathcal{F}_{k}(t)$
uniform bound are uniformly bounded on any fixed interval (by the
uniformity in Theorem \ref{thm:fluct of bergman}) and \ref{eq:vanishing of der}
the theorem now follows by integrating over $t.$ Indeed, since $\mathcal{F}_{k}(t)$
and its first derivative vanish at $t=0$ we have 
\[
\mathcal{F}_{k}(t)=\int\int\frac{d^{2}\mathcal{F}_{k}(s)}{d^{2}t}\chi(v,s)dvds,
\]
where $\chi$ is the characteristic function of the set of all $(v,s)$
such that $v\leq s\leq t.$ Hence \ref{eq:ptwise conv of second deriv}
gives 
\begin{equation}
\mathcal{F}_{k}(t)\rightarrow a\int\int\chi(v,s)dvds=a\frac{t^{2}}{2},\,\,a:=a:=-\left\Vert du\right\Vert _{dd^{c}\phi}^{2}\label{eq:pt-wise asym in proof clt}
\end{equation}
 which proves the point-wise version of the asymptotics \ref{eq:asympt of Laplace tr in Thm intro}
when $t\in\R.$ 

Next, we set $\nu_{k}:=k^{-(n-1)/2}\widetilde{\mathcal{N}_{k}}[u]_{*}(\gamma_{k}),$
which gives a sequence of compactly supported probability measures
on $\R,$ obtained by pushing forward the probability measure $\gamma_{k}.$
Then we may write 
\[
\mathcal{F}_{k}(t)=\int_{\R}\nu_{k}(s)e^{-ts}
\]
which gives a well-defined holomorphic function for all $t$ in $\C$
with 
\[
\left|f_{k}(t)\right|\leq C_{R}
\]
 for all $t\in\C$ such that $\left|t\right|\leq R.$ By \ref{eq:pt-wise asym in proof clt}
we have $f_{k}(t)\rightarrow f(t),$ where $f(t)$ is an entire function,
on the maximally totally real set $\R$ in $\C.$ Hence, the same
normal families argument as below formula \ref{eq:pf of thm bulk univ}
shows that uniform convergence actually holds on compacts of $\C$
(even for all derivatives). Setting $t=i\xi$ with $\xi\in\R$ in
particular gives that the Fourier transforms $\widehat{\nu_{k}}$
converges uniformly om compacts in $\R_{\xi}$ towards $\widehat{\nu,}$
where $\widehat{\nu}$ (and hence $\nu)$ is a centered Gaussian.
As is well-known this latter convergence property is equivalent to
convergence in distribution.
\end{proof}
Finally, the variance asymptotics then follows by evaluating the convergence
of the second derivatives at $t=0$ and using lemma \ref{lem:formula for exp and var}.

\subsection{Proof of Cor \ref{cor:clt} (the normalized CLT)}

The case when $u$ is supported in the bulk follows directly from
Theorem \ref{thm:conv of laplace}. Next, we recall that by \cite{s2}
the normalized CLT holds, for a general determinantal point processes,
under the condition that $\mbox{Var}\ensuremath{(\mathcal{N}(u))\rightarrow\infty}$(as
$N\rightarrow\infty)$ and that there exists a positive numbers $\delta$
and $C$ such that 
\[
\E(\mathcal{N}(u))\leq C\left(\mbox{Var \ensuremath{(\mathcal{N}(u))}}\right)^{\delta}.
\]
 Since $\E(\mathcal{N}(u))\sim N\sim k^{n}$ the validity of these
assumptions in the present setting, when $n\geq2,$ follows directly
from the lower bound on the variance in Theorem \ref{thm:conv of laplace}
(by taking $\delta=(n-1)/n).$

\subsection{\label{sub:An-alternative-proof}An alternative proof of the CLT
for smooth data using second order expansions}

We start by recalling the following result in \cite{berm45} generalizing
the seminal asymptotic expansion of Zelditch and Catlin concerning
the case when $dd^{c}\phi>0$ on all of $X$ (see \cite{ze,b-b-s}). 
\begin{thm}
\label{thm:sec order exp}Assume that $\phi$ is a smooth weight on
the ample line bundle $L,$ $\omega_{n}$ a smooth volume form on
$X$ and $\phi_{F}$ a smooth metric on a line bundle $F.$ Then,
on the diagonal, the point-wise norm of the Bergman kernel $K_{k}$
of $H^{0}(X,kL+F)$ endowed with the corresponding $L^{2}-$norm admits
a complete asymptotic expansion on any compact subset of bulk. More
precisely, the corresponding second order expansion is given by 
\[
|K_{k}(x,x)|_{k\phi+\phi_{F}}\frac{\omega^{n}}{n!}=
\]

\end{thm}
\[
=\frac{k^{n}}{n!}\omega_{\phi}^{n}+\frac{k^{n-1}}{(n-1)!}\left(-\frac{1}{2}\mbox{Ric \ensuremath{\omega_{\phi}}+}\mbox{Ric}\ensuremath{\omega}+dd^{c}\phi_{F}\right)\wedge\omega_{\phi}^{n-1}+O(k^{n-2}),
\]
(the form $\mbox{Ric\,\ensuremath{\eta}}:=-dd^{c}\log\eta^{n}$ represents
the normalized Ricci curvature of a Kähler metric $\eta).$
\begin{rem}
Strictly speaking the result in \cite{berm45} was only formulated
when $F$ is trivial (which in fact will be enough for our purposes).
But exactly the same proof applies for a general $F.$ Indeed, around
any point where $\omega_{\phi}>0$ \cite{b-b-s} gives the expansion
for a local version of the Bergman kernel (the contribution to the
coefficients coming from the line bundle $F$ are computed in \cite[Section 2.5]{b-b-s}).
Then the local Bergman kernel is shown to coincide with the global
one in the bulk using Theorem \ref{thm:horm-kod-for phi} with $L$
replaced by $kL+F-K_{X}$ (just as in the proof of Step 2 in Lemma
\ref{lem:lower bd}). 
\end{rem}
In particular, by the previous theorem the following holds in the
bulk: 
\begin{equation}
\left(|K_{k}(x,x)|_{k\phi+\phi_{F}}-K_{k}(x,x)|_{k\phi}\right)\frac{\omega^{n}}{n!}=\frac{k^{n-1}}{(n-1)!}dd^{c}\phi_{F}\wedge(dd^{c}\phi)^{n-1}+O(k^{n-2}),\label{eq:difference of bergman meas}
\end{equation}
Let us now specialize to the case when $n=1$ and apply the previous
result to the trivial line bundle $F$ endowed with the weight $\phi_{F}=tu$
for $t\in\R$ and $u$ a smooth function supported in the interior
of the bulk. Then it is not hard to see that the remainder term above
is uniform in $t,$ as long as $|t|\leq C$ (indeed, the proof in
\cite{b-b-s} shows that the remainder term only depends on an upper
bound on the local sup-norm of the local derivatives of $\phi_{F}$). 

Now, combining the asymptotics in \ref{eq:difference of bergman meas}
with the first formula in Lemma \ref{lem:formula for exp and var}
gives 

\[
-\frac{d}{dt}\log\E_{k\phi+tu}(e^{-t\mathcal{\tilde{N}}[u]})=\int_{X}|K_{k}(x,x)|_{k\phi+tu}u\omega-\int_{X}|K_{k}(x,x)|_{k\phi}u\omega=
\]
\[
=t\int_{X}(udd^{c}u+o(1),
\]
 where the remainder term tends to zero, uniformly in $k$ and $t.$
Hence, integrating over $t$ gives

\[
-\log\E_{k\phi+tu}(e^{-t\mathcal{\tilde{N}}[u]})=\int_{0}^{t}sds\int_{X}udd^{c}u=\frac{1}{2}\int_{X}udd^{c}u,
\]
proving the asymptotics in formula \ref{eq:asympt of Laplace tr in Thm intro}
in this special case (which implies Theorem \ref{thm:conv of laplace},
just as before). In fact, the uniformity in $t$ used above may be
dispensed with. Indeed, by the convexity of $t\mapsto g(t):=-\log\E_{k\phi+tu}(e^{-t\mathcal{\tilde{N}}[u]})$
we have $g'(0)\leq g'(t)\leq g'(1)$ so that the dominated convergence
theorem may be applied. 
\begin{rem}
\label{rem:sec order asympt of expect}It follows immediately from
Theorem \ref{thm:sec order exp} that, for $u$ as above, the expectation
of $\mathcal{N}(u)$ has a complete asymptotic expansion of the form
\[
\E(\mathcal{N}(u))=\int_{X}u\left(\frac{k^{n}}{n!}\omega_{\phi}^{n}+\frac{k^{n-1}}{(n-1)!}\left(-\frac{1}{2}\mbox{Ric \ensuremath{\omega_{\phi}}+}\mbox{Ric}\ensuremath{\omega}\right)\wedge\omega_{\phi}^{n-1}\right)+O(k^{n-2}).
\]
Moreover, when $\omega_{\phi}>0$ on all of $X$ integrating the asymptotics
in Theorem \ref{thm:sec order exp} yields a complete asymptotic expansion
of the partition function $\log Z_{N_{k}}[\phi]$ corresponding to
$(\phi,\omega_{n})$ (see the notation Section \ref{sub:Relations-to-phase}):
\[
-\frac{1}{N_{k}k}\log Z_{N_{k}}[\phi]=\mathcal{F}_{0}[\phi]+\mathcal{F}_{1}[\phi]k^{-1}+...,
\]
 where $\mathcal{F}_{0}$ and $\mathcal{F}_{1}$ are explicit functionals,
well-known in Kähler geometry ($\mathcal{F}_{0}$ is the primitive
$\mathcal{E}$ of the Monge-Ampère operator, sometimes called the
Aubin-Yau energy and $\mathcal{F}_{1}$ is a twisted version of the
K-energy functional \cite{don1}).
\end{rem}
It seems likely that a similar argument applies when $n>1,$ using
$\phi_{F}=k^{(n-1)/2}t.$ But then one has to verify that the remainder
terms are uniform in $k.$ Alternatively, one could, at least formally,
apply the \emph{first }order asymptotics of $K_{k}(x,x)|_{k\tilde{\phi}}$
with the \emph{perturbed }weight 
\begin{equation}
\tilde{\phi}:=\phi+k^{-1}k^{(n-1)/2}u\label{eq:pert weight}
\end{equation}
Indeed, setting $\phi_{t}:=\phi+tu,$ handling the limit $k\rightarrow\infty$
formally gives
\[
k^{-(n-1)/2}\left(K_{k}(x,x)|_{k\tilde{\phi}}-K_{k}(x,x)_{k\phi}\right)\approx\frac{d}{dt}_{|t=0}k^{-n}K_{k}(x,x)_{k\phi_{t}}\approx
\]
\[
\approx\frac{d\mu_{\phi_{t}}}{dt}_{|t=0}=\frac{1}{(n-1)!}dd^{c}u\wedge(dd^{c}\phi)^{n-1}
\]
Anyway, an important feature of the proof of Theorem \ref{thm:conv of laplace}
in the previous section is that it only requires that $u$ be Lipschitz
continuous. In contrast, any argument based on the second order expansion
in Theorem \ref{thm:sec order exp} requires that $u$ be, at least,
$C^{2}-$smooth, ensuring that $\Delta u$ is point-wise defined. 
\begin{rem}
\label{rem:ward}The alternative proof above is similar to the method
of proof in the real setting in \cite{j} and the second proof of
the corresponding result in \cite{a-h-m1}, also concerning the case
$n=1$ (the first proof in \cite{a-h-m3} uses the method of cumulants).
The second proof, which was only sketched in \cite{a-h-m1}, uses
the formal first order argument involving the perturbed weight $\tilde{\phi}$
above which was made rigorous in \cite{a-h-m3}, for real analytic
$\phi,$ using the method of Ward identities. An important feature
of the method in \cite{a-h-m3} is that it also applies on the boundary
of $S$ giving the precise ``edge contribution''. It would be very
interesting to extend the results in \cite{a-h-m3} (and the generalizations
in \cite{l-s,b-b-n-y}) to the case when $n>1,$ as further discussed
in the following section. 
\end{rem}

\section{\label{sec:Outlook-on-relations}Outlook on relations to LDPs and
phase transitions }

\subsection{From the LDP towards a general CLT}

Let us start with some general considerations. Consider an $N-$particle
random point processes $(\mu^{(N)},X^{N})$ on a compact topological
space $X.$ Assume that the law of the corresponding empirical measure
\[
\delta_{N}:=\frac{1}{N}\sum_{i=1}^{N}\delta_{x_{i}}
\]
satisfies a large deviation principle (LDP) at a speed $r_{N}\rightarrow\infty$
and rate functional $E(\mu)$ on $\mathcal{P}(X),$ symbolically expressed
as 
\[
(\delta_{N})_{*}\mu^{(N)}\sim e^{-r_{N}E(\mu)},\,\,N\rightarrow\infty
\]
(see \cite{d-z} for the precise meaning of a LDP). In particular,
by the contraction principle, this implies a LDP at the same speed
$r_{N}$ for the real-valued random variable $\left\langle \delta_{N},u\right\rangle $
on $(\mu^{(N)},X^{N})$ defined by a given continuous function $u\in C^{0}(X).$
It is well-known that, in general, a LDP at a speed $r_{N}$ for a
real-valued random variable implies, under suitable further assumptions
(that are unfortunately rather strong) a CLT of the following form:

\begin{equation}
r_{N}^{1/2}(\left\langle \delta_{N},u\right\rangle -\E(\left\langle \delta_{N},u\right\rangle )\rightarrow\mathcal{N}(0,\sigma_{u}),\label{eq:general form of clt}
\end{equation}
in distribution, where the variance $\sigma_{u}$ is given by 
\begin{equation}
\sigma_{u}=-\frac{d^{2}\mathcal{F}(tu),}{d^{2}t}_{|t=0},\label{eq:general form for variance}
\end{equation}
expressed in terms of the concave functional $\mathcal{F}(u)$ defined
by the following limit: 
\begin{equation}
\mathcal{F}(u):=\lim_{N\rightarrow\infty}\mathcal{F}^{(N)}(u),\,\,\,\mathcal{F}^{(N)}(u):=-\log\E(e^{-r_{N}\left\langle u,\delta_{N}\right\rangle }),\label{eq:asymptotics of free energy}
\end{equation}
where $\frac{1}{r_{N}}\log\E(e^{-r_{N}t\left\langle u,\delta_{N}\right\rangle })$
is thus a scaling of the moment generating function $\log\E(e^{-\left\langle u,\delta_{N}\right\rangle })$
of the random variable $\left\langle u,\delta_{N}\right\rangle .$
The existence of the limit above follows from the LDP (by Varadhan's
lemma \cite{d-z}) and the functional $\mathcal{F}$ on $C^{0}(X)$
coincides with the Legendre-Fenchel transform of the rate functional
$E(\mu).$ For example, by \cite{Br}, the CLT follows from the LDP
under the assumption that $f(t):=\mathcal{F}(tu)$ is real-analytic
and the convergence of $\mathcal{F}^{(N)}(tu)$ towards $f(t)$ can
be extended to complex valued $t$ (which, in particular, requires
the absence of phase transitions at any order, as recalled below). 

Conversely, we make the following simple observation:
\begin{prop}
If the LDP holds with a speed $r_{N}$ and a CLT (as in formula \ref{eq:general form of clt})
holds, then the corresponding variance $\sigma_{u}$ is given by 
\[
\sigma_{u}=-\lim_{N\rightarrow\infty}\frac{d^{2}\mathcal{F}^{(N)}(tu),}{d^{2}t}_{|t=0}.
\]
\end{prop}
\begin{proof}
If the CLT holds then 
\[
g^{(N)}(t):=\log\E(e^{-(r_{N})^{1/2}(\left\langle u,\delta_{N}\right\rangle -\E\left\langle u,\delta_{N}\right\rangle )})\rightarrow a|t|^{2}/2
\]
 in the $C_{loc}^{\infty}-$topology, where $a\in\R$ is the corresponding
variance (by the argument used in the end of the proof of Theorem
\ref{thm:conv of laplace}). In particular,
\[
\frac{d^{2}g^{(N)}(t)}{d^{2}t}_{|t=0}\rightarrow a.
\]
But, $g^{(N)}(t)=-r_{N}f^{(N)}(r_{N}^{-1/2}t)+\E(\left\langle u,\delta_{N}\right\rangle )t$
and hence $\frac{d^{2}g^{(N)}(tu)}{d^{2}t}_{|t=0}$ coincides with
$-\frac{d^{2}f^{(N)}(t)}{d^{2}t}_{|t=0},$ which concludes the proof. 
\end{proof}
In the present setting the LDP for the laws of the empirical measure
is established in \cite{berman ldp} at a speed 
\[
r_{N}=kN_{k}
\]
and the corresponding functional $\mathcal{F}$ (formula \ref{eq:asymptotics of free energy})
may be expressed as 

\[
\mathcal{F}(u)=\mathcal{E}((\phi+u)_{e}),
\]
 where $\mathcal{E}$ is a primitive of complex Monge-Ampère operator,
i.e. for any smooth weight $\phi$ and smooth function $u$ 
\[
\frac{\mathcal{E}((\phi+tu))}{dt}_{|t=0}=\frac{1}{n!}\int_{X}(dd^{c}\phi)^{n}u
\]
Moreover, by \cite[Thm B]{b-b}, the functional $\mathcal{F}$ is
Gateaux differentiable on $C^{0}(X)$ and its differential at $\phi$
is represented by the corresponding equilibrium measure, i.e. for
any $u\in C^{0}(X)$ 
\begin{equation}
\frac{d\mathcal{F}(tu)}{dt}_{|t=0}=\frac{1}{n!}\int_{X}(dd^{c}\phi_{e})^{n}u\label{eq:first deriv of free energ}
\end{equation}
Since the linear statistic $\mathcal{N}[u]$ is given by 
\[
\mathcal{N}[u]:=N\left\langle u,\delta_{N}\right\rangle 
\]
 and $N\sim k^{n}$ the general discussion above thus suggests that,
under suitable assumptions, a CLT of the following form should hold:
\[
N^{-(1-1/n)/2}(\mathcal{N}[u]-\E(\mathcal{N}[u])\rightarrow\mathcal{N}(0,\sigma_{u}),
\]
which is thus consistent with the CLT in Theorem \ref{thm:conv of laplace}
and Corollary \ref{cor:clt}.
\begin{rem}
As shown in \cite{berman ldp}, the LDP in the present setting follows
from the asymptotics \ref{eq:asymptotics of free energy} (established
in the present setting in \cite[Thm A]{b-b}) together with the Gärtner-Ellis
theorem, using the differentiability of $\mathcal{F}.$ The corresponding
rate functional $E$ on $\mathcal{P}(X)$ may then be defined as the
Legendre-Fenchel transform on $\mathcal{P}(X)$ of the functional
$\mathcal{F}$ and the differentiability of $\mathcal{F}$ corresponds
to the strict convexity of $E$ (on the convex subset $\{E<\infty\}\subset\mathcal{P}(X)).$
In fact, the LDP in \cite{berman ldp} holds in the very general setting
where $\mu$ has the property that $(\phi,\mu)$ satisfies the Bernstein-Markov
property for any continuous weight $\phi$ (i.e. the corresponding
one-point correlation density has sub-exponential growth). In particular,
this is the case in the purely real setting where $X=\R^{n}$ and
$\phi$ has super logarithmic growth. 
\end{rem}

\subsection{\label{sub:Relations-to-phase}Relations to phase transitions}

In the present setting the probability measure $\mu^{(N)}$ on $X^{N}$
may be represented as the Gibbs measure 
\[
\mu^{(N)}:=\frac{e^{-\beta E^{N}}}{Z_{N}[\phi]}\mu_{0}^{\otimes N},\,\,\,Z_{N}[\phi]:=\int_{X^{N}}e^{-\beta E^{N}}\mu^{\otimes N}
\]
at inverse temperature $\beta=2,$ of the Hamiltonian 
\[
E^{(N)}:=-\log\left|\det(S)(x_{1},...,x_{N})\right|_{k\phi}
\]
 where $Z_{N}[\phi]$ is the corresponding partition function (see
Remark \ref{rem:gram}). Accordingly, the scaled moment generating
function may, in the terminology of statistical mechanics, be represented
as a difference of scaled\emph{ free energies}:
\[
\frac{1}{r_{N}}\log\E(e^{-r_{N}t\left\langle u,\delta_{N}\right\rangle })=\frac{1}{kN_{k}}\log Z_{N}[\phi+tu]-\frac{1}{kN_{k}}\log Z_{N}[\phi].
\]
The limiting functional $\mathcal{F}(u)$ can thus be viewed as the
thermodynamical free energy functional, describing the leading asymptotics
of the $N-$dependent free energies $\mathcal{F}^{(N)}(u)$, as $N\rightarrow\infty.$
We recall that, according to Ehrenfest's classical classification
of phase transitions, a system is said to exhibit a \emph{phase transition
of order $m$} when the $m$ th derivative of the thermodynamical
free energy has a discontinuity when considering variations of the
thermodynamical variable in question (assuming that the lower order
derivatives exist and are continuous). In the present setting the
thermodynamical variable is the function $u$ defining the linear
statistic and we have the following 
\begin{prop}
Given a smooth bounded function $u\in C^{0}(X)$ the thermodynamical
free energy $t\mapsto\mathcal{F}(tu)$ has continuous first order
derivatives. Moreover, the right and left second order derivatives
exist at $t=0$ and are given by 
\begin{equation}
\frac{d^{2}\mathcal{F}(tu),}{d^{2}t}_{|t=0^{\pm}}=\frac{1}{(n-1)!}\int v_{\pm}dd^{c}u\wedge(dd^{c}\phi_{e})^{n-1}\label{eq:form sec order deriv in thm}
\end{equation}
where the right and left derivatives 
\begin{equation}
v_{\pm}:=\frac{d(\phi+tu)_{e}}{dt}_{|t=0^{\pm}}\label{eq:def of v plus minus}
\end{equation}
exist, defining bounded functions on $X.$ \end{prop}
\begin{proof}
As recalled above the existence of the \emph{first} order derivatives
when $X$ is compact is the content of \cite[Thm B]{b-b} and the
superlogarithmic setting when $X=\C^{n}$ is shown in \cite{berman ldp}.
In order to study the second order derivatives first observe that
$t\mapsto(\phi+tu)_{e}(x)$ is concave (indeed it is defined as the
sup of linear functions). In particular, it is locally Lipschitz continuous
and hence the right and left derivatives $v_{\pm},$ at $t=0,$ indeed
exist and are in $L^{\infty}$. Now, fixing $t\neq0$ and setting
$\psi_{t}:=(\phi+tu)_{e}$ we have, by formula \ref{eq:first deriv of free energ},
\[
\frac{d\mathcal{F}(tu)}{dt}-\frac{d\mathcal{F}(0)}{dt}=\int_{X}u\left((dd^{c}\psi_{t})^{n}-(dd^{c}\psi_{0})^{n}\right)/n!.
\]
 Expanding the bracket and integrating by parts this means that
\[
t^{-1}(\frac{d\mathcal{F}(tu)}{dt}-\frac{d\mathcal{F}(0)}{dt})=\int_{X}dd^{c}u\wedge t^{-1}(\psi_{t}-\psi_{0})\left((dd^{c}\psi_{t})^{n-1}...+(dd^{c}\psi_{0})^{n-1}\right)/n!.
\]
By the regularity results in\cite{berm45,berm4} $dd^{c}\psi_{t}$
is a $L^{\infty}-$current which is uniformly bounded in $t$ (for
bounded $t)$ and by concavity the left and right limits $v_{\pm}$
of $t^{-1}(\psi_{t}-\psi_{0})$ as $t\rightarrow0^{\pm}$ exist and
are monotonic in $t.$ Hence, applying the dominated convergence theorem
proves formula \ref{eq:form sec order deriv in thm}. 
\end{proof}
This means that there is an absence of first order phase transitions
in the present setting. In the light of the discussion in the previous
section it is tempting to speculate that the linear statistic corresponding
to a smooth bounded function $u$ on $X$ satisfies a CLT, as in formula,
if one assumes that $\frac{d(\phi+tu)_{e}}{dt}_{|t=0}$ exists, i.e.
\[
v_{+}=v_{-}
\]
(perhaps with additional regularity assumptions) and that the limit
$\sigma_{u}$ of the scaled variances $N^{1/n-1}\mbox{Var}\mathcal{N}(u)$
is then given by
\begin{equation}
\frac{d^{2}\mathcal{F}(tu)}{d^{2}t}_{|t=0}=-\frac{1}{(n-1)!}\int\frac{d(\phi+tu)_{e}}{dt}_{|t=0^{\pm}}dd^{c}u\wedge(dd^{c}\phi_{e})^{n-1}\label{eq:limiting variance in higher dim}
\end{equation}
In the case when $u$ is supported in the interior of the bulk this
is consistent with Theorem \ref{thm:conv of laplace}. Indeed, then
$v_{\pm}=u$ and an integration by parts thus reveals that the integral
above coincides with the variance in question. The speculation above
is also consistent with the results in \cite{a-h-m3,l-s,b-b-n-y}
concerning the setting of super logarithmic growth in $\C.$  Indeed,
in the most general results appearing in \cite{l-s,b-b-n-y} it is,
in particular, assumed that $\Delta\phi>0$ on a neighborhood of the
support $S$ and that the boundary of the support has no singular
points (cusps) in the sense of \cite{c-r}. Under these assumptions
it can be shown that $\frac{d(\phi+tu)_{e}}{dt}_{|t=0}$ exists and
is given by the function $\tilde{u}$ defined as $u$ on $S$ and
on $X-S$ as the harmonic extension of $u.$ The point is that, assuming
that the support $S_{\phi_{t}}$ varies continuously with $t,$ the
following holds in the complement of $S:$ 
\[
0=\frac{d\mu_{\phi_{t}}}{dt}_{|t=0}=dd^{c}\frac{d(\phi+tu)_{e}}{dt}_{|t=0}
\]
 In particular, one then has 
\[
\frac{d^{2}\mathcal{F}(tu),}{d^{2}t}_{|t=0^{\pm}}=-\int_{X}\tilde{u}dd^{c}\tilde{u}=\int_{X}d\tilde{u}\wedge d^{c}\tilde{u,}
\]
which indeed coincides with the formula for the variance in \cite{a-h-m3,l-s,b-b-n-y}.
It would be very interesting to extend the CLTs in \cite{a-h-m3,l-s,b-b-n-y}
to higher dimensions $n>1$ and show that the limiting variance is
given by formula \ref{eq:limiting variance in higher dim}. Under
the regularity assumption that $\phi_{e}$ admits a Monge-Ampère foliation
by Riemann surfaces in the complement $S^{c}$ the role of $\tilde{u}$
is then played by the extension of $u$ which is harmonic along the
leaves $\mathcal{L}_{\alpha}$ of the foliation and 
\[
\frac{d^{2}\mathcal{F}(tu),}{d^{2}t}_{|t=0^{\pm}}=-\int d\alpha\int_{\mathcal{L}_{\alpha}}d\tilde{u}\wedge d^{c}\tilde{u},
\]
 i.e. a certain superposition of the Dirichlet norms of $\tilde{u}$
along the leaves. Even though the regularity assumption used above
is rather strong (in general it holds if $\phi_{e}\in C_{loc}^{3}(S^{c})$
and $(dd^{c}\phi_{e})^{n-1}$ is of rank $n-1$ in $S^{c})$ there
are certainly particular geometrically settings where it is satisfied.
For example, it applies in the setting of \cite{RWN} and in the equivariant
settings in \cite{p-s,r-s,z-z}. 

Even if the limit of the scaled variances $N^{1/n-1}\mbox{Var}\mathcal{N}(u)$
may not exist for a general strongly regular weighted measure $(\phi,\omega_{n})$
it seems natural to expect that the sequence is always bounded. By
Lemma \ref{lem:formula for exp and var} this would follow from the
validity of the following 
\begin{conjecture}
Given a strongly regular weighted measure $(\phi,\mu)$ there exists
a constant $C$ such that 
\[
\frac{1}{2}\int_{X\times X}k^{-(n-1)}\left|K_{k}(x,y)\right|_{k\phi}^{2}d(x,y)^{2}\mu\otimes\mu\leq C,
\]
 where $d(x,y)$ is the distance function corresponding to a given
metric on $X.$
\end{conjecture}
In the ``real setting'', i.e. case when $\mu$ is supported on a
real algebraic variety (or on $X:=\R^{n}$ in the super logarithmic
setting) the estimate in the previous conjecture was established in
\cite{b-o} (in the case $X=\R$ with $\phi$ real analytic the estimate
is shown in \cite{pa1b}). Moreover, by the second item in Prop \ref{pro:varians est}
a weaker form of the conjecture holds, where the constant $C$ is
replaced by $o(k).$


\begin{thebibliography}{10}
\bibitem{b-v-}Alvarez-Gaumé, L; Bost, J-B; Moore, G; Nelson, P; Vafa,
C: Bosonization on higher genus Riemann surfaces. Comm. Math. Phys.
112 (1987), no. 3, 503--552. 

\bibitem{a-h-m1}Ameur, Y; Hedenmalm, H; Makarov, N: Fluctuations
of eigenvalues of random normal matrices. Duke Math. J. 159 (2011),
no. 1, 31\textendash 81. arXiv:0807.0375. 

\bibitem{a-h-m2}Ameur, Y; Hedenmalm, H; Makarov, N: Berezin transform
in polynomial Bergman spaces. Comm. Pure Appl. Math. 63 (2010), no.
12. arXiv:0807.0369 

\bibitem{a-h-m3}Ameur, Y; Hedenmalm, H; Makarov, N: Random normal
matrices and Ward identities. Ann. Probab. 43 (2015), no. 3, 1157\textendash 1201.
arXiv:1109.5941

\bibitem{a-k-m}Ameur, Y; Kang, NG; Makarov, N: Rescaling Ward identities
in the random normal matrix model. arXiv:1410.4132, 2014 

\bibitem{b-h}Bardenet, R; Hardy,A: Monte Carlo with Determinantal
Point Processes. arXiv:1605.00361

\bibitem{b-b-n-y}Bauerschmidt, R; Bourgade, P; Nikula, M; Yau, H-T:
The two-dimensional Coulomb plasma: quasi-free approximation and central
limit theorem. arXiv:1609.08582

\bibitem{b-t}Bedford, E; Taylor, A: The Dirichlet problem for a complex
Monge-Ampere equation. Invent. Math 37 (1976), no 1, 1-44

\bibitem{bvg}Berline, N; Getzler, E; Vergne, M: Heat kernels and
Dirac operators. Corrected reprint of the 1992 original. Grundlehren
Text Editions. Springer-Verlag, Berlin, 2004. 

\bibitem{berm1}Berman, R.J: Bergman kernels and local holomorphic
Morse inequalities. Math Z., Vol 248, Nr 2 (2004), 325--344 

\bibitem{berm2}Berman, R.J: Super Toeplitz operators on holomorphic
line bundles J. Geom. Anal. 16 (2006), no. 1, 1--22.

\bibitem{berm3}Berman, R.J: Bergman kernels and equilibrium measures
for polarized pseudoconcave domains. Internat. J. Math. 21 (2010),
no. 1, 77\textendash 115.

\bibitem{berm4}Berman, R.J: Bergman kernels and weighted equilibrium
measures of $\C^{n}.$ Indiana Univ.Math. Journal, Volume 58, issue
4, 2009

\bibitem{berm45}Berman, R.J: Bergman kernels and equilibrium measures
for line bundles over projective manifolds. The American Journal of
Mathematics, Volume 131, Number 5, October 2009

\bibitem{b-b-s}Berman R.J; Berndtsson B; Sjöstrand J: A direct approach
to asymptotics of Bergman kernels for positive line bundles. Arkiv
för Matematik. Volume 46 (2008) no. 2, 197--217 

\bibitem{b-b}Berman, R.J; Boucksom, S; Growth of balls of holomorphic
sections and energy at equilibrium\emph{.} 42 pages, Invent. Math.
181 (2010), no. 2, 337-394 

\bibitem{b-b-w}Berman, R.J; Boucksom, S; Witt Nyström, D: Fekete
points and convergence towards equilibrium measures on complex manifolds\emph{.}
Acta Math. Vol. 207, Issue 1 (2011), 1-27, 

\bibitem{berman ldp}Berman, R.J: Determinantal point processes and
fermions on complex manifolds: Large deviations and Bosonization.
Comn. in Math. Physics 2014, Volume 327, Issue 1, pp 1-47, arXiv:0812.4224. 

\bibitem{berman4 och halv}Berman, R.J: Sharp Asymptotics for Toeplitz
Determinants and Convergence Towards the Gaussian Free Field on Riemann
Surfaces. International Mathematics Research Notices, 2012

\bibitem{berm5}Berman, R.J: Kähler-Einstein metrics, canonical random
point processes and birational geometry. http://arxiv.org/abs/1307.3634
(to appear in the AMS Proceedings of the 2015 Summer Research Institute
on Algebraic Geometry).

\bibitem{b-o}R.J. Berman, J. Ortega-Cerdà:\emph{ }Sampling of real
multivariate polynomials and pluripotential theory\emph{.} arXiv:1509.00956.
American J. of Math (to appear).

\bibitem{b-s-z}Bleher, P; Shiffman, B; Zelditch, S: Universality
and scaling of correlations between zeros on complex manifolds. Inventiones
Mathematicae, 2000. 

\bibitem{bo}Bogaevski\u{\i}, I. A. Singularities of convex hulls
of three-dimensional hypersurfaces. Proc. Steklov Inst. Math. 1998,
no. 2 (221), 71\textendash 90 

\bibitem{b-d-e}Bonnet, G., David. F., and Eynard, B., Breakdown of
universality in multi-cut matrix models, J. Phys. A33, 6739-6768 (2000)

\bibitem{b-s}Boutet de Monvel; Sjötrand, J: Sur la singularite des
noyaux de Bergman et de Szegö. Asterisque 34 \textendash{} 35 (1976),
123\textendash 164

\bibitem{Br}Bryc, W: A remark on the connection between the large
deviation principle and the central limit theorem. Statistics \& probability
letters, 1993 - Elsevier

\bibitem{c-r}Caffarelli, L. A.; Rivière, N. M: Smoothness and analyticity
of free boundaries in variational inequalities. Ann. Scuola Norm.
Sup. Pisa Cl. Sci. (4) 3 (1976), no. 2, 289\textendash 310.

\bibitem{c-k-s}Cooper, F; Khare, A: Sukhatme, U: Supersymmetry in
Quantum Mechanics. World Scientific Publ. 2001

\bibitem{dei1}Deift, P.A: Universality for mathematical and physical
systems. International Congress of Mathematicians. Vol. I, 125--152,
Eur. Math. Soc., Zürich, 2007.

\bibitem{dei2}Deift, P. A. Orthogonal polynomials and random matrices:
a Riemann-Hilbert approach. Courant Lecture Notes in Mathematics,
3. New York University, Courant Institute of Mathematical Sciences,
New York; American Mathematical Society, Providence, RI, 1999. 

\bibitem{d-}Deift, P., Kriecherbauer, T., McLaughlin, K.T.-R., Venakides,
S. and Zhou, X., Uniform asymptotics for polynomials orthogonal with
respect to varying exponential weights and applications to universality
questions in random matrix theory. Comm. Pure Appl. Math. v52. 

\bibitem{del}Delin, H: Pointwise estimates for the weighted Bergman
projection kernel in $\C^{n}$ using a weighted $L^{2}$ estimate
for the $\bar{\partial}$ equation. Ann. Inst. Fourier (Grenoble)
48 (1998), no. 4, 967--997. 

\bibitem{de1}Demailly, J-P: Estimations $L^{2}$ pour l'opérateur
$\bar{\partial}$ d'un fibré vectoriel holomorphe semi-positif au-dessus
d'une variété kählérienne complète. (French). Ann. Sci. École Norm.
Sup. (4) 15 (1982), no. 3, 457--511. 

\bibitem{de4}Demailly, J-P: Complex analytic and algebraic geometry.
Available at www-fourier.ujf-grenoble.fr/\textasciitilde{}demailly/books.html

\bibitem{de3}Demailly, J-P: Potential Theory in Several Complex Variables.
Manuscript available at www-fourier.ujf-grenoble.fr/\textasciitilde{}demailly/ 

\bibitem{d-z}Dembo, A; Zeitouni O: Large deviation techniques and
applications. Corrected reprint of the second (1998) edition. Stochastic
Modelling and Applied Probability, 38. Springer-Verlag, Berlin, 2010.
xvi+396 pp. 

\bibitem{don1}Donaldson, S. K. Scalar curvature and projective embeddings.
II. Q. J. Math. 56 (2005), no. 3, 345--356.

\bibitem{f1}Forrester, P. J. Particles in a magnetic field and plasma
analogies: doubly periodic boundary conditions. J. Phys. A 39 (2006),
no. 41, 13025--13036.

\bibitem{f2}Forrester, P. J. Fluctuation formula for complex random
matrices. J. Phys. A 32 (1999), no. 13, L159--L163.

\bibitem{gi}Ginibre,J: Statistical ensembles of complex, quaternion,
and real matrices. J. Mathematical Phys. 6 (1965), 440--449;

\bibitem{gr-ha}Griffiths, P; Harris, J: Principles of algebraic geometry.
Wiley Classics Library. John Wiley \& Sons, Inc., New York, 1994.

\bibitem{g-z}Guedj,V; Zeriahi, A: Intrinsic capacities on compact
Kähler manifolds. J. Geom. Anal. 15 (2005), no. 4, 607--639.

\bibitem{gu}Guionnet, A: Large deviations and stochastic calculus
for large random matrices. Probab. Surv. 1 (2004), 72--172 (electronic). 

\bibitem{gu-m-s}Gurbatov, S.N: Malakhov, A.I: Saichev, A.I: Non-linear
random waves and turbulence in non-dispersive media: Waves, Rays,
Particles. Manchester Univ. Press, Manchester 1991. With an appendix
(``Singularities and bifurcations of potential flows'') by Arnold
et al. 

\bibitem{g-m-s}Götz, M.; Maymeskul, V. V.; Saff, E. B. Asymptotic
distribution of nodes for near-optimal polynomial interpolation on
certain curves in $\R^{2}$. Constr. Approx. 18 (2002), no. 2, 255--283.

\bibitem{fkz}Ferrari, F; Klevtsov, S; Zelditch, S: Random Kähler
metrics. Nuclear Phys. B 869 (2013), no. 1, 89\textendash 110.

\bibitem{h-m}Hedenmalm, H; Makarov, N: Quantum Hele-Shaw flow, Preprint
in 2004 at arXiv.org/abs/math.PR/0411437

\bibitem{h-k-p}Hough, J. B.; Krishnapur, M.; Peres, Y.l; Virág, B:
Determinantal processes and independence. Probab. Surv. 3 (2006),
206--229 

\bibitem{j}Johansson, K: On fluctuations of eigenvalues of random
Hermitian matrices. Duke Math. J. 91 (1998), no. 1, 151--204. 

\bibitem{j2}Johansson, K: Random matrices and determinantal processes.
arXiv:math-ph/0510038 

\bibitem{kle}Klevtsov, S: Geometry and large N limits in Laughlin
states. arXiv:1608.02928 

\bibitem{kl}Klimek, M: Pluripotential theory. London Mathematical
Society Monographs. New Series, 6. Oxford Science Publications. The
Clarendon Press, Oxford University Press, New York, 1991 

\bibitem{lau}Laughlin, RB: Elementary theory: the incompressible
quantum fluid. In ``The Quantum Hall Effect'', 1987 - Springer

\bibitem{la}Lazarsfeld, R: Positivity in algebraic geometry. I. Classical
setting: line bundles and linear series. II. Positivity for vector
bundles, and multiplier ideals. A Series of Modern Surveys in Mathematics,
48. and 49. Springer-Verlag, Berlin, 2004.

\bibitem{l-s}Leblé, T; Serfaty, S: Fluctuations of Two-Dimensional
Coulomb Gases. arXiv:1609.08088

\bibitem{li}Lindholm, N: Sampling in weighted $L^{p}$ spaces of
entire functions in $\C^{n}$ and estimates of the Bergman kernel,
J. Funct. Anal. 182 (2001), 390-426.

\bibitem{m}Macchi, O: The coincidence approach to stochastic point
processes. Advances in Appl. Probability 7 (1975), 83--122.

\bibitem{pa1}Pastur, L. A simple approach to the global regime of
Gaussian ensembles of random matrices. Ukraïn. Mat. Zh. 57 (2005),
no. 6, 790--817; translation in Ukrainian Math. J. 57 (2005), no.
6, 936--966 

\bibitem{pa1b}Pastur, L: Limiting laws of linear eigenvalue statistics
for Hermitian matrix models. J. Math. Phys. 47 (2006), no. 10, 

\bibitem{pa2}Pastur, L.; Shcherbina, M: Bulk universality and related
properties of Hermitian matrix models. J. Stat. Phys. 130 (2008),
no. 2, 205--250.

\bibitem{p-s}Pokorny, FT; Singer, M: Toric partial density functions
and stability of toric varieties. Mathematische Annalen, 2014 - Springer

\bibitem{r-v0}Rider, B; Virág, B: The noise in the circular law and
the Gaussian free field. Int. Math. Res. Not. IMRN 2007, no. 2,

\bibitem{r-v2}Rider,B; Virag, B: Complex determinantal processes
and H1 noise. Electronic Journal of Probability. Vol. 12 (2007) 

\bibitem{r-s}Ross, J; Singer, M: Asymptotics of Partial Density Functions
for Divisors, arXiv:1312.1145

\bibitem{RWN}Ross, J; Witt Nyström, D: Homogeneous Monge-Ampère Equations
and Canonical Tubular Neighbourhoods in Kähler Geometry. arXiv:1403.3282.

\bibitem{s-t}Saff.E; Totik.V: Logarithmic potentials with exteriour
fields. Springer-Verlag, Berlin. (1997) (with an appendix by Bloom,
T)

\bibitem{Sc}Schaeffer,D: Some examples of singularities in a free
boundary. Annali della Scuola Normale Superiore di Pisa - Classe di
Scienze (1977) Volume: 4, Issue: 1, page 133--144.

\bibitem{t-s-z}Scardicchio, A; Torquato, S; Zachary, C.E: Point processes
in arbitrary dimension from fermionic gases, random matrix theory,
and number theory . J. Stat. Mech. Theory Exp. 2008, no. 11, 

\bibitem{t-s-z2}Scardicchio, A; Torquato, S; Zachary, C.E: Statistical
properties of determinantal point processes in high-dimensional Euclidean
spaces. Phys. Rev. E (3) 79 (2009), no. 4.

\bibitem{she}Sheffield, Scott: Gaussian free fields for mathematicians.
Probab. Theory Related Fields 139 (2007), no. 3-4, 521--541.

\bibitem{sz2}Shiffman, B; Zelditch, S: Distribution of zeros of random
and quantum chaotic sections of positive line bundles. Comm. Math.
Phys. 200 (1999), no. 3, 661--683.

\bibitem{sz4}Shiffman, B; Zelditch S: Number variance of random zeros
on complex manifolds, II: smooth statistics. Pure Appl. Math. Q. 6
(2010), no. 4, Special Issue: In honor of Joseph J. Kohn. Part 2.

\bibitem{sh}Shigekawa, I.: Spectral properties of Schrodinger operators
with magnetic fields for a spin 1/2 particle. 101 (2): 255-285 (1991)

\bibitem{s-w}Sloan, I.H.; Womersley, R.S: Extremal systems of points
and numerical integration on the sphere. Adv. Comput. Math. 21 (2004),
no. 1-2, 107--125.

\bibitem{s1}Soshnikov, A. Determinantal random point fields. (Russian)
Uspekhi Mat. Nauk 55 (2000), no. 5(335), 107--160; translation in
Russian Math. Surveys 55 (2000), no. 5, 923--975

\bibitem{s2}Soshnikov, A: Gaussian limit for determinantal random
point fields. Ann. Probab. 30 (2002), no. 1, 171--187. 

\bibitem{za}Zabrodin, A; Matrix models and growth processes: from
viscous flows to the quantum Hall effect. arXiv.org/abs/hep-th/0411437.
NATO Sci. Ser. II Math. Phys. Chem., 221, Springer, Dordrecht, 2006. 

\bibitem{ze}Zelditch, S: Szegö kernels and a theorem of Tian. Internat.
Math. Res. Notices 1998, no. 6, 317--331.

\bibitem{z-z}Zelditch, S; Zhou, P: Interface asymptotics of partial
Bergman kernels on S1-symmetric Kaehler manifolds\end{thebibliography}
\end{document}